\documentclass[11pt, a4paper, draft]{amsart}
\usepackage{amsmath, amsthm, amssymb}
\usepackage[english]{babel}
\usepackage[utf8]{inputenc}
\usepackage[all]{xy}
\usepackage{xspace}
\usepackage{comment}
\usepackage{setspace}
\usepackage{enumerate}
\usepackage{mathtools}

\usepackage[mathscr]{eucal}
\usepackage{hyperref}
\usepackage{graphicx}
\usepackage{tikz}
\usetikzlibrary{matrix}

	\newcommand{\Prim}{\ensuremath{\operatorname{Prim}}}

	\newcommand{\Ad}{\ensuremath{\operatorname{Ad}}\xspace}
	\renewcommand{\min}{\mathrm{min}}

	\newcommand{\Ann}{\mathrm{Ann}\,}

	\newcommand{\dimnuc}{\mathrm{dim}_{\mathrm{nuc}}}
	
	\newcounter{theorem}
	\theoremstyle{plain}
	
	\newtheorem{lemma}[theorem]{Lemma}
	\newtheorem{theorem}[theorem]{Theorem}
	\newtheorem{proposition}[theorem]{Proposition}
	\newtheorem{corollary}[theorem]{Corollary}

	\newtheorem{sublemma}[theorem]{Sublemma}
	
	\theoremstyle{definition}
	\newtheorem{definition}[theorem]{Definition}
	
	\newtheorem{remark}[theorem]{Remark}
	
	\newtheorem{example}[theorem]{Example}
	\newtheorem{problem}[theorem]{Problem}

	\numberwithin{equation}{section}
	\numberwithin{figure}{section}

\newcommand{\id}{\mathrm{id}}
\newcommand{\dr}{\mathrm{dr}}

\newcommand{\Cu}{\mathrm{Cu}}

\numberwithin{equation}{section}

\newcommand{\Cstar}{{\ensuremath C^*}}
\newcommand{\mCstar}{{\ensuremath C^*}}

\newcommand{\seq}{(\infty)}

\title[The nuclear dimension of $\mathcal O_\infty$-stable $\Cstar$-algebras]{The nuclear dimension of $\mathcal O_\infty$-stable $\Cstar$-algebras}

\begin{document}
\author{Joan Bosa}
\address{\hskip-\parindent Joan Bosa, Departamento de Matemáticas, Universidad de Zaragoza, C/Pedro Cerbuna 12, Zaragoza, 50009, Spain.} \email{jbosa@unizar.es}
\author{James Gabe}
\address{\hskip-\parindent James Gabe, Department of Mathematics and Computer Science, University of Southern Denmark, Odense, Denmark.}
\email{gabe@imada.sdu.dk}
\author{Aidan Sims}
\address{\hskip-\parindent Aidan Sims, School of Mathematics and Applied Statistics, University of Wollongong, Wollongong NSW 2522, Australia.}
\email{asims@uow.edu.au}
\author{Stuart White}
\address{\hskip-\parindent Stuart White, Mathematical Institute, University of Oxford, Radcliffe Observatory Quarter, Woodstock Road, Oxford, OX2 6GG, United Kingdom}
\email{stuart.white@maths.ox.ac.uk}

\thanks{Research partially supported by an Alexander von Humboldt Foundation Fellowship (SW), Australian Research Council grant DP180100595 (AS), a Carlsberg Foundation Internationalisation Fellowship (JG), by the DGI-MINECO and European Regional Development Fund through grant MTM2017-83487-P (JB), EPSRC:EP/R025061/1 (SW), and the Beatriu de Pinós postdoctoral programme of the Government
	of Catalonia’s Secretariat for Universities and Research 2017-BP-0079 (JB)}
\begin{abstract}
We show that every nuclear $\mathcal O_\infty$-stable $^*$-homomorphism with a separable exact
domain has nuclear dimension at most $1$. In particular separable, nuclear, $\mathcal
O_\infty$-stable $\Cstar$-algebras have nuclear dimension $1$. We also characterise when
$\mathcal O_\infty$-stable $\Cstar$-algebras have finite decomposition rank in terms of
quasidiagonality and primitive-ideal structure, and determine when full $\mathcal O_2$-stable
$^*$-homomorphisms have nuclear dimension~0.
\end{abstract}
\maketitle

\section*{Introduction}

\renewcommand*{\thetheorem}{\Alph{theorem}}

Nuclear dimension for $\Cstar$-algebras, introduced in \cite{WZ:Adv}, provides a non-commutative analogue
of Lebesgue covering dimension. Via the Gelfand transform, every commutative
$\Cstar$-algebra $A$ is isomorphic to the algebra $C_0(X)$ of continuous functions vanishing at
infinity on some locally compact Hausdorff space $X$, and then the nuclear dimension of $A$ is
exactly the Lebesgue covering dimension of $X$. Simple $\Cstar$-algebras lie at the other extreme.
These are highly non-commutative, and here nuclear dimension provides the dividing line between
the tame and the wild.  Indeed, separable, simple, unital, $\Cstar$-algebras of finite nuclear
dimension satisfying the Universal Coefficient Theorem (UCT) of Rosenberg and Schochet
\cite{RS:DMJ} are now classified by their Elliott invariants
(\cite{K:Book,P:Doc,GLN,EGLN,TWW:Ann}; this precise statement can be found as
\cite[Corollary~D]{TWW:Ann}).

Accordingly there has been substantial interest in determining the nuclear dimension of
$\Cstar$-algebras, with a heavy initial focus on Kirchberg algebras (the simple, separable,
nuclear and purely infinite algebras classified by independently Kirchberg and Phillips using
Kasparov's bivariant $K$-theory). There have been two complementary strands of development.
Kirchberg algebras satisfying the UCT were shown to have nuclear dimension one in \cite{RSS:Adv},
building on the earlier work \cite{RST:Adv,E:JFA,WZ:Adv}. As it remains a major open problem
whether all separable nuclear $\Cstar$-algebras (or equivalently all Kirchberg algebras) satisfy
the UCT, Matui and Sato instigated a new approach to the nuclear dimension computation of Kirchberg algebras in
\cite{MS:DMJ} which does not rely on the UCT. A remarkable theorem of Kirchberg \cite{K:ICM} (see
also \cite{KP:crelle}), characterises these now eponymous algebras amongst simple separable
nuclear $\Cstar$-algebras: $A$ is Kirchberg if and only if it tensorially absorbs the Cuntz
algebra $\mathcal O_\infty$ in the sense that $A\cong A\otimes\mathcal O_\infty$. Peeling off a
tensor factor of $\mathcal O_\infty$ creates extra space that makes the analysis in \cite{MS:DMJ}
possible. This approach ultimately led to a proof that all Kirchberg algebras have
nuclear dimension~1  (\cite{BBSTWW}). Tensorial absorption of other strongly self-absorbing $\Cstar$-algebras has
also been used to compute nuclear dimension for many finite simple $\Cstar$-algebras in
\cite{ENST:arXiv,SWW:Invent,BBSTWW,CETWW,CE}.

Tensorial absorbtion of $\mathcal O_\infty$ also has profound implications for non-simple
$\Cstar$-algebras, including Kirchberg and R\o{}rdam's algebraic characterisation of $\mathcal
O_\infty$-stable nuclear $\Cstar$-algebras (\cite{KR-adv}) and Kirchberg's classification of
separable nuclear $\mathcal O_\infty$-stable $\Cstar$-algebras via ideal-related bivariant
$K$-theory (\cite{K:Book}) as outlined in \cite{K:German}. This has played an important role in the
classification of non-simple Cuntz--Krieger algebras and their automorphisms (\cite{CRR:TAMS,
ERR:JNCG}) and other non-simple algebras with small ideal lattices (\cite{ARR:JFA}). So it is
natural to seek to compute the nuclear dimension of $\mathcal O_\infty$-stable $\Cstar$-algebras.
In \cite{S:Adv}, Szab\'o showed that separable $\mathcal O_\infty$-stable algebras have nuclear
dimension at most~$3$ (extending the special cases from \cite{BEMSW}) but the precise value
remained open. Our first main result settles this question:
\begin{theorem}\label{ThmA}
Let $A$ be a separable, nuclear $\mathcal O_\infty$-stable $\Cstar$-algebra.  Then $A$ has nuclear
dimension $1$.
\end{theorem}

Nuclear dimension was preceded by the earlier notion of decomposition rank from \cite{KW:IJM}. The
difference between the two conditions appears small, but is significant. It is tied up with the
notion of quasidiagonality, which is an external approximation property originating in work of
Halmos \cite{Popa}. Since quasidiagonality implies stable finiteness, Kirchberg algebras are never
quasidiagonal. However every $\Cstar$-algebra with finite decomposition rank is quasidiagonal
\cite[Equation~(3.3) and Proposition~5.1]{KW:IJM}. So all Kirchberg algebras have nuclear
dimension 1, but infinite decomposition rank. Indeed, since finite decomposition rank passes to
quotients, if $A$ has finite decomposition rank, then every quotient of $A$ is quasidiagonal. So,
for example, while $C_0((0,1], \mathcal{O}_\infty)$ has finite nuclear dimension (this dates back to \cite{WZ:Adv})
and is quasidiagonal by Voiculescu’s famous homotopy invariance of quasidiagonality \cite{Voi91},
it has infinite decomposition rank because it surjects onto $\mathcal{O}_\infty$. It is perhaps
surprising at first sight that there can exist $\mathcal O_\infty$-stable $\Cstar$-algebras of
finite decomposition rank,  but R\o{}rdam's $\mathcal O_\infty$-stable approximately subhomogeneous
algebra from \cite{Rordam-IJM} provides an example.

Quasidiagonality has repeatedly played a major role in the structure and classification theory for
$\Cstar$-algebras. For example, in hindsight Winter's use in \cite{W:PLMS} of finite decomposition rank to
access tracial approximations, and thence Lin's classification results (\cite{L:DMJ}), for
$\Cstar$-algebras of real rank zero, hinges on quasidiagonality. This became explicit
in Matui and Sato's work \cite{MS:DMJ}, which planted the seeds for the use of quasidiagonality in
the classification results of \cite{EGLN,TWW:Ann}. We now know that for simple $\Cstar$-algebras,
quasidiagonality delineates between finite nuclear dimension and finite decomposition rank;
building on \cite{MS:DMJ,SWW:Invent,BBSTWW}, it is shown in \cite{CETWW} that a simple unital
$\Cstar$-algebra with finite nuclear dimension has finite decomposition rank precisely when it and
all its traces are quasidiagonal (the latter in the sense of \cite{B:MAMS}).

Our second main result completely characterises when nuclear $\mathcal O_\infty$-stable
$\Cstar$-algebras have finite decomposition rank. Although quasidiagonality need not pass to
quotients, finite decomposition rank does, so a necessary condition for finite decomposition rank
is that every quotient is quasidiagonal. We prove that this necessary condition is also
sufficient. Moreover using \cite{Gabe:new}, we can also describe when these algebras have finite
decomposition rank in terms of the primitive ideal space.

\begin{theorem}\label{ThmB}
Let $A$ be a separable, nuclear $\mathcal O_\infty$-stable $\Cstar$-algebra. Then the following
are equivalent:
\begin{enumerate}[(i)]
\item $A$ has finite decomposition rank;
\item $A$ has decomposition rank $1$;
\item All quotients of $A$ are quasidiagonal;
\item Every non-zero hereditary $\Cstar$-subalgebra of $A$ is stable;
\item The primitive ideal space of $A$ has no locally closed, one point subsets.
\end{enumerate}
\end{theorem}

In order to classify separable nuclear $\mathcal O_\infty$-stable $\Cstar$-algebras Kirchberg
classified morphisms between these algebras. More generally, he classified nuclear
$^*$-homomorphisms from separable exact $\Cstar$-algebras into $\mathcal O_\infty$-stable
$\Cstar$-algebras. Although originally defined for $\Cstar$-algebras, the definitions of both
decomposition rank and nuclear dimension of a $\Cstar$-algebra $A$ are given in terms of
approximation properties for the identity map $\id_A \colon A \to A$, and make perfect sense for other
$^*$-homomorphisms between $\Cstar$-algebras (see \cite{TW:APDE}). Following the philosophy that
we should detect classifiability through nuclear dimension, it is then reasonable to ask
whether the maps classified by Kirchberg are of finite nuclear dimension.

Motivated by Kirchberg's transfer of the nuclearity hypothesis from $\Cstar$-algebras to morphisms in his
classification of nuclear maps between non-nuclear $\Cstar$-algebras, one should seek to do likewise with
other ingredients of classification theorems. In particular, in his new treatment
\cite{Gabe17} of Kirchberg's $\mathcal O_2$-stable classification theorem, the second named author
introduces $\mathcal O_\infty$ and $\mathcal O_2$-stable $^*$-homomorphisms. These include all
$^*$-homomorphisms whose domain or codomain is $\mathcal O_\infty$-stable or $\mathcal O_2$-stable
respectively, as well as considerably more examples. He shows that nuclear $\mathcal O_2$-stable
maps out of separable exact $\Cstar$-algebras are classified by their behaviour on ideals. The
forthcoming work \cite{G:InPrep} will establish a version of Kirchberg's classification theorem
for $\mathcal O_\infty$-stable maps. Our third main theorem establishes finite nuclear dimension
for such maps; in particular, in answer to the question above, the maps classified by Kirchberg
all have nuclear dimension at most~1.

\begin{theorem}\label{ThmC}
Let $A$ and $B$ be $\Cstar$-algebras with $A$ separable and exact. Let $\theta \colon A\rightarrow B$ be
a nuclear $\mathcal O_\infty$-stable $^*$-homomorphism. Then $\theta$ has nuclear dimension at
most~$1$. If every quotient of $A$ is quasidiagonal, then the decomposition rank of $\theta$ is at
most~$1$.
\end{theorem}

In addition to recognising classifiability via nuclear dimension, we regard this theorem as a
proof of concept for the development of regularity theory for $^*$-homomorphisms in the spirit of
the Toms--Winter conjecture (see \cite{ET:BAMS,W:Invent1,WZ:Adv}). In particular, in the setting
of stably finite $\Cstar$-algebras, the Jiang--Su algebra, $\mathcal Z$, from \cite{JS:AJM} is the
appropriate replacement for $\mathcal O_\infty$, and a key future goal is to develop a suitable
notion of Jiang--Su stability for maps, and an analogue of Winter's $\mathcal Z$-stability theorem
from \cite{W:Invent1,W:Invent2} in this context.

The results discussed above show that very large classes of maps have nuclear dimension either
0~or~1; but which ones have dimension~0? For spaces, it is more or less immediate that zero
dimensionality coincides with total disconnectedness; and for $\Cstar$-algebras Winter established
early in the theory that zero dimensionality coincides with approximate finite dimensionality (see
\cite{Winter-JFA}, noting that this uses an earlier notion of non-commutative covering dimension).
For maps, characterising zero dimensionality seems harder, but we do so for full $\mathcal
O_2$-stable maps, in terms of quasidiagonality.

\begin{theorem}\label{ThmD}
Let $A$ and $B$ be $\Cstar$-algebras with $A$ separable and exact, and suppose that $\theta \colon
A \to B$ is a full, $\mathcal O_2$-stable $^\ast$-homomorphism. Then the nuclear dimension of
$\theta$ is given by
\begin{equation}
\dimnuc\theta = \left\{ \begin{array}{ll}
0 & \textrm{ if $\theta$ is nuclear and $A$ is quasidiagonal}, \\
1 & \textrm{ if $\theta$ is nuclear and $A$ is not quasidiagonal},\\
\infty&\text{ otherwise}.
\end{array} \right.
\end{equation}
\end{theorem}

\medskip
\subsection*{\sc Discussion of methods}\hfill \\

All existing techniques for passing from tensorial absorption to finite topological dimension rely on some aspect of classification. In \cite{WZ:Adv}, Winter and Zacharias estimate the nuclear
dimension of Cuntz algebras by direct analysis. They combine this with R\o{}rdam's models for
UCT-Kirchberg algebras as inductive limits of algebras of the form $\bigoplus_{i=1}^k
M_{m_i}\otimes\mathcal O_{n_i}\otimes C(\mathbb T)$ (obtained using the Kirchberg--Phillips
classification) to show that UCT-Kirchberg algebras have nuclear dimension at most~$5$. This
approach was refined in \cite{E:JFA} and \cite{RST:Adv}, leading to the realisation in
\cite{RSS:Adv} of UCT-Kirchberg algebras as inductive limits of higher rank graph algebras with
nuclear dimension~$1$.

Matui and Sato's use of $\mathcal O_\infty$-absorption to obtain finite nuclear dimension absent
the UCT also uses classification, in the form of uniqueness results for maps. Since $\mathcal
O_\infty$ is strongly self-absorbing, to compute the nuclear dimension of an $\mathcal
O_\infty$-stable $\Cstar$-algebra $A$, it suffices to compute the nuclear dimension of the map
$\mathrm{id}_A\otimes 1_{\mathcal O_\infty} \colon A\rightarrow A\otimes\mathcal O_\infty$. Matui
and Sato do this for Kirchberg algebras in \cite{MS:DMJ} using a trade off between $\mathcal
O_\infty$-stability and Kirchberg's $\mathcal O_2$-embedding theorem. They approximate
$\id_A\otimes 1_{\mathcal O_\infty}$ with two copies of an embedding $A\hookrightarrow \mathcal
O_2\hookrightarrow 1_A\otimes\mathcal O_\infty\subseteq A\otimes\mathcal O_\infty$; classification
results enter implicitly through a $2\times 2$ matrix trick \`a la Connes. This strategy was made
both more general and more explicit in \cite[Theorem~3.3]{BEMSW}. Szab\'o's work \cite{S:Adv} also
uses a trade off of this nature; he uses R\o{}rdam's strongly purely infinite approximately
subhomogeneous algebra $\mathcal A_{[0,1]}$ in place of $\mathcal O_2$, obtaining a bound on
nuclear dimension for arbitrary $\mathcal O_\infty$-stable nuclear $\Cstar$-algebras.  These
methods yield approximations which naturally decompose into $4$ summands (via approximations for
$\mathcal O_2$ and $\mathcal A_{[0,1]}$ that each decompose into two summands), so yield nuclear
dimension at most $3$.\footnote{Recall that approximations that decompose into $n+1$ summands
witness nuclear dimension at most $n$; see Definition \ref{DefNDim} below.}

To obtain the exact nuclear dimension value of $1$, we must avoid passing through
intermediate one dimensional algebras like $\mathcal O_2$, and instead build factorisations
through zero dimensional building blocks directly. The analysis of \cite[Section~9]{BBSTWW}
achieves this for Kirchberg algebras in four broad steps. First use a positive contraction $h\in
\mathcal O_\infty$ of spectrum $[0,1]$ to decompose $\id_A\otimes 1_{\mathcal O_\infty}$ as the
sum of two order zero maps $\mathrm{id}_A\otimes h$ and $\mathrm{id}_A\otimes (1_{\mathcal
O_\infty}-h)$. Next use Voiculescu's quasidiagonality of the cone over $A$ from \cite{Voi91} to
obtain a sequence $(\phi_n\colon A\rightarrow F_n)_{n=1}^\infty$ of approximately order zero maps into
finite dimensional algebras $F_n$. Next embed each $F_n$ into $A\otimes\mathcal O_\infty$ using
the second tensor factor, and thereby regard each $\phi_n$ as a map $A\rightarrow A\otimes\mathcal
O_\infty$. Finally employ ingredients from Kirchberg's classification of simple purely infinite
nuclear $\Cstar$-algebras to power a $2\times 2$ matrix trick to see that the sequence
$(\phi_n)_{n=1}^\infty$ is approximately unitarily equivalent to each of $\mathrm{id}_A\otimes h$
and $\mathrm{id}_A\otimes (1_{\mathcal O_\infty}-h)$. This gives the required finite-dimensional
approximations.

\medskip
The broad strategy of the previous paragraph underpins our proof of Theorem~\ref{ThmA}. However
when $A$ is non-simple, both the construction of finite dimensional models and the appropriate
classification vehicle for transferring these models to $\mathrm{id}_A\otimes h$ are more
delicate.\footnote{By a trick from \cite{BEMSW}, $h$ and $1_{\mathcal O_\infty}-h$ are
approximately unitarily equivalent in $\mathcal O_\infty$, so that $\mathrm{id}_A\otimes h$ and
$\mathrm{id}_A\otimes (1_{\mathcal O_\infty}-h)$ are approximately unitarily equivalent. So it
suffices to model $\id_A\otimes h$, and we suppress $\mathrm{id}_A\otimes (1_{\mathcal
O_\infty}-h)$ henceforth.}   For the latter, we must use classification results for non-simple
algebras. The order zero map $\id_A\otimes h$, gives rise to a $^*$-homomorphism $\pi \colon f\otimes a
\mapsto a\otimes f(h)$ from the cone $C_0((0,1])\otimes A$ into $A\otimes \mathcal O_\infty$.
Since embeddings of cones into Kirchberg algebras are $\mathcal O_2$-stable (see
Lemma~\ref{Lm:HO2}), the second author's classification framework for $\mathcal O_2$-stable maps
in \cite{Gabe17} applies to show that $\pi$ is determined up to approximate equivalence by
its behaviour on ideals.

To obtain our finite dimensional models, we construct a sequence of diagrams
\begin{equation}\label{IntroEq2}
\xymatrix{A\ar[rr]\ar[dr]_{\tilde\psi_n}&&A\otimes\mathcal O_\infty\\&F_n\ar[ur]_{\eta_n}}
\end{equation}
in which the $F_n$ are finite dimensional $\Cstar$-algebras, the $\eta_n$ are contractive order
zero maps, and the $\tilde\psi_n$ are completely positive contractions and approximately order
zero (in a point-norm sense). The sequence $(\eta_n\circ\tilde{\psi}_n)_{n=1}^\infty$ induces an
order zero map $\theta$ from $A$ into the sequence algebra $\prod_i A / \bigoplus_i A$, which is
typically denoted $A_\infty$, but here denoted $A_{\seq}$ because $\mathcal{O}_\infty$ features
heavily throughout. (For readers unfamiliar with sequence algebras, these are described just
before Lemma~\ref{l:fctcalcseqoz}). Through the duality (\cite{WZ:MJM}) between order zero maps on
$A$ and $^*$-homomorphisms out of $C_0((0,1])\otimes A$, this $\theta$ is determined by a
$^*$-homomorphism $\rho\colon C_0((0,1])\otimes A\rightarrow (A\otimes\mathcal O_\infty)_{\seq}$.
Our goal is to choose the diagrams~\eqref{IntroEq2} so that $\rho$ is $\mathcal O_2$-stable and
induces the same map on ideals as $\pi$ (viewed as a map into $(A\otimes\mathcal
O_\infty)_{\seq})$. Classification will then ensure that $\rho$ and $\pi$, and hence also the
associated order zero maps, are approximately equivalent.\footnote{The form of approximate
equivalence is subtle; classification gives approximate Murray and von Neumann equivalence of
$\rho$ and $\pi$, which in turn gives approximate unitary equivalence after passing to a $2\times
2$ matrix amplification.  We perform the nuclear dimension calculations in this matrix
amplification, and subsequently compress back to the original $\Cstar$-algebra $A$.} This provides
the required finite-dimensional approximations for $\id_A\otimes h$, and hence shows that $A$ has
nuclear dimension one.

The construction of the diagrams~\eqref{IntroEq2} occupies the bulk of the paper.\footnote{We note
for comparison with Lemmas \ref{lm:Main}~and~\ref{l:Oinftymapmain} (which provide the meat of the
construction) that for technical reasons we actually construct approximate $^*$-homomorphisms
$\psi_n\colon C_0((0,1])\otimes A\rightarrow F_n$, and then the maps $\tilde{\psi}_n$
in~(\ref{IntroEq2}) are given by $\tilde{\psi}_n(a)=\psi_n(\id_{(0,1]}\otimes a)$. In the main
body of the paper, we make no explicit reference to $\tilde\psi_n$.} When $A$ has exactly one
nontrivial ideal $I$, the idea is to take the `downward maps' $\tilde\psi_n$ as a direct sum
$\tilde\psi_{n,1}\oplus \tilde\psi_{n,2}\colon A\rightarrow F_{n,1}\oplus F_{n,2}$ where
$(\tilde\psi_{n,1})_{n=1}^\infty$ arises from order zero maps corresponding to the
quasidiagonality of the cone over $A$ and $(\tilde\psi_{n,2})_{n=1}^\infty$ by following the
quotient $A\rightarrow A/I$ by maps obtained from the quasidiagonality of the cone over $A/I$. The
`upward maps' $\eta_n$ are obtained by fixing embeddings $\iota_{n,i}\colon F_{n,i}\rightarrow\mathcal
O_\infty$ and orthogonal\footnote{$\mathcal O_\infty$-stability facilitates this orthogonality.}
positive contractions $h_I$ and $h_A$ which are full in $I$ and $A$ respectively, and setting
$\eta_n(x_1,x_2)=h_I\otimes \iota_1(x_1)+h_A\otimes \iota_2(x_2)$. Then elements $x\in I$ are
annihilated under the sequence $(\tilde\psi_{n,2})_{n=1}^\infty$, and so
$(\eta_n\circ\tilde{\psi}_n(x))_{n=1}^\infty$ sees only the first component, which generates the
ideal $I$.  Keeping track of exactly which ideal is generated in a sequence algebra turns out to
be a little delicate; we do this in Lemma~\ref{lm:Main}, where we also need an additional smearing
across the interval to ensure that the resulting map $\rho$ has the required behaviour on all
ideals of $C_0((0,1])\otimes A$, not just those of the form $C_0((0,1])\otimes I$ for some $I\lhd
A$.\footnote{This is the point behind the composition with the multiplication map $m^*$
in~(\ref{eq:rhocomp}).}  It remains to arrange $\mathcal O_2$-stability, which we do in
Section~\ref{S4}, using a similar smearing to merge the $\rho$ from Section~\ref{S3} with an
embedding of a cone into $\mathcal O_\infty$.

To handle general ideal lattices, we simultaneously model the ideal lattice of $A$ by finite
sublattices, while performing a similar construction to that outlined above. The machinery we need
to approximate the ideal lattice of $A$ by finite sublattices is set up in Section~\ref{S2}, with
technical difficulties arising throughout our main arguments because the approximating finite
sublattice is typically not linearly ordered.

With more care, the strategy discussed above also takes care of the nuclear dimension part of
Theorem~\ref{ThmC}. As it turns out, the decomposition rank component of Theorem~\ref{ThmC} is a
little easier in that we can use the quasidiagonality hypothesis directly in place of quasidiagonality
of cones in our main technical construction.  We review nuclear dimension and decomposition rank in Section~\ref{S1}, setting
out the criteria we use to recognise finite decomposition rank.  This enables us to handle both the nuclear dimension and decomposition rank cases of Theorem~\ref{ThmC} in tandem, keeping track of the required extra detail throughout Sections \ref{S3}~and~\ref{S4} and the proof of Theorem~\ref{ThmC} in Section~\ref{S5}.

Section~\ref{S6} turns to the characterisation of finite decomposition rank, using
Theorem~\ref{ThmC} and the developments of \cite{Gabe:new} to prove Theorem~\ref{ThmB}.   We end in
Section~\ref{S7} with the proof of Theorem~\ref{ThmD}. With Theorem~\ref{ThmC} in place, the last
component of Theorem~\ref{ThmD} is obtained by using quasidiagonality of $A$ to produce finite
dimensional models for a full map $\theta$; the fullness assumption
ensures that the map carries no ideal data, and then $\mathcal O_2$-stable classification ensures
that these models can actually be used to approximate $\theta$.

\medskip

\paragraph*{\textbf{Acknowledgments}} Portions of this research where undertaken during the research programme on the classification of Operator Algebras: Complexity, Rigidity and Dynamics at the Institut Mittag--Leffler in 2016, and the intensive research programme on Operator Algebras: Dynamics and Interactions at CRM, Barcelona in 2017.  We thank the organisers and funders of these programmes.
We are also grateful for the valuable suggestions made by the anonymous referee.

\renewcommand*{\thetheorem}{\roman{theorem}}
\numberwithin{theorem}{section}

\section{Nuclear dimension and decomposition rank}\label{S1}

In this section we recall the definitions of the nuclear dimension and decomposition rank from
\cite{WZ:Adv} and \cite{KW:IJM} respectively, and collect some relatively standard results for
later use.  Throughout the paper if $a,b$ are elements of a $\Cstar$-algebra, and $\epsilon>0$, we
write $a\approx_\epsilon b$ to mean $\|a-b\|\leq\epsilon$.

Recall that a completely positive map $\phi\colon A\to B$ between $\Cstar$-algebras is said to be
\emph{order zero} if, for every $a,b \in A_+$ with $ab=0$, one has $\phi(a)\phi(b)=0$. Building on
\cite{Wolff}, the structure theory for these maps was developed in \cite{WZ:MJM}.  In particular,
there is a $^*$-homomorphism $\pi_\phi \colon A \to \mathcal M (\mCstar(\phi(A)))$ (called the
\emph{supporting $^*$-homomorphism}), taking values in the multiplier algebra $\mathcal M(\mCstar(\phi(A)))$ of $\mCstar(\phi(A))$, and a positive contraction $h \in \mathcal
M(\mCstar(\phi(A)))\cap \pi(A)'$ such that
\begin{equation}\label{eq:hpioz}
\phi(a) = h \pi(a) , \quad a\in A.
\end{equation}
This then induces a $^*$-homomorphism $\rho_\phi \colon C_0((0,1]) \otimes A \to \mCstar(\phi(A))
\subseteq B$ such that
\begin{equation}\label{eq:rhooz}
\rho_\phi(f \otimes a) = f(h) \pi(a), \quad f\in C_0((0,1]),\ a\in A.
\end{equation}

Conversely, each $^*$-homomorphism from $C_0((0,1]) \otimes  A$ to $B$ determines an order zero map from
$A$ to $B$ by composition with the completely positive and contractive (cpc) order zero inclusion $a \mapsto \id_{(0,1]} \otimes a$ of
$A$ into its cone.

\begin{proposition}[{\cite[Corollary~3.1]{WZ:MJM}}]\label{prop.Cone}
Let $A,B$ be $\Cstar$-algebras. There is a one-to-one correspondence between cpc, order zero maps $\phi\colon A \to B$ and ${}^*$-homomorphisms $\rho\colon C_0((0,1])
\otimes A \to B$, where $\phi$ and $\rho$ are related by the commuting diagram
\begin{equation}
\xymatrix{
A \ar[rr]^-{a \mapsto \mathrm{id}_{(0,1]} \otimes a} \ar[drr]_-{\phi} && C_0((0,1]) \otimes A \ar[d]^{\rho} \\
&& B \, .
}
\end{equation}
\end{proposition}

The structural theorem yields a functional calculus on order zero maps (cf.
\cite[Corollary~4.2]{WZ:MJM}): if $\phi \colon A \to B$ is a cpc order zero map, and  $f\in
C_0((0,1])_+$ is a contraction, then $f(\phi) \colon A\to B$ is the cpc~order zero map such that
\begin{equation}\label{fncaloz}
f(\phi)(a) = \rho_\phi(f\otimes a), \quad a\in A.
\end{equation}

We record two facts regarding induced $^*$-homomorphisms on the cones and functional calculus for
later use.  These are proved by routine verification of the defining properties of
$\rho_{\eta\circ\psi}$ and $f(\phi)$ respectively.
\begin{lemma}\label{l:compoz}
Let $\psi \colon A \to F$ be a $^*$-homomorphism, and let $\eta \colon F \to B$ be a cpc~order
zero map. The $^*$-homomorphism $\rho_{\eta\circ\psi}\colon C_0((0,1]) \otimes A \to B$ induced by the
composition $\eta \circ \psi$ is exactly the composition
\begin{equation}
C_0((0,1]) \otimes A \xrightarrow{\id_{C_0((0,1])} \otimes \psi } C_0((0,1]) \otimes F \xrightarrow{\rho_{\eta}} B,
\end{equation}
where $\rho_\eta$ is the $^*$-homomorphism induced by $\eta$.
\end{lemma}

\begin{lemma}\label{Lem:cpcRanges}
Let $\phi_1,\dots, \phi_n \colon A \to B$ be cpc~order zero maps with pairwise orthogonal ranges,
and let $f\in C_0((0,1])_+$ be a contraction. Then $\phi \coloneqq \sum_{i=1}^n \phi_i \colon A
\to B$ is a cpc~order zero map, and
\begin{equation}
f(\phi) = \sum_{i=1}^n f(\phi_i) \colon A \to B.
\end{equation}
\end{lemma}

Throughout the paper we often use sequence algebras to encode approximate behaviour.  Given a
sequence $(B_n)_{n=1}^\infty$ of $\Cstar$-algebras, the \emph{sequence algebra} is the quotient of
the $\ell^\infty$-product by the $c_0$-sum: $\prod_{n=1}^\infty B_n/\bigoplus_{n=1}^\infty B_n$.
In the special, but particularly relevant, case where $B_n = B$ for all $n$, it is common to denote
this sequence algebra by $B_\infty$. However, in this paper the Cuntz algebra $\mathcal{O}_\infty$
plays a major role, so the subscripted $\infty$ could cause confusion, and the natural alternative,
namely $B_\omega$ (where $\omega$ is intended to denote the co-finite filter on $\mathbb{N}$), is
problematic because we reserve $\omega$ to denote a free ultrafilter on $\mathbb{N}$ in Section~\ref{S7}.
Instead, we have adopted the slightly unusual notation
$B_{\seq} := \prod_{n=1}^\infty B/\bigoplus_{n=1}^\infty B$.
The $C^*$-algebra $B$ embeds into $B_{\seq}$ as (cosets of) constant
sequences, and we will often implicitly regard $B$ as a $\Cstar$-subalgebra of $B_{\seq}$.
When needed, we will denote the embedding by $\iota_B\colon B\to B_{\seq}$.

We denote elements in a sequence algebra $\prod_{n=1}^\infty B_n/\bigoplus_{n=1}^\infty B_n$ by
representatives in $\prod_{n=1}^\infty B_n$, rather than introducing additional notation for the
class they represent. By the same principle, given cpc maps $\phi_n \colon A \to B_n$, we often write
$(\phi_n)^\infty_{n=1}$ for the induced map $A_{\seq} \to \prod^\infty_{n=1}
B_n / \bigoplus^\infty_{n=1} B_n$ and also for its restriction to $A$. For later use we record the following routine fact regarding
functional calculus of sequences of order zero maps, which does not seem to have explicitly
appeared in the literature.

\begin{lemma}\label{l:fctcalcseqoz}
Let $(\eta_n\colon F_n\to B_n)_{n=1}^\infty$ be a sequence of cpc order zero maps between
$\Cstar$-algebras, and let
\begin{equation}
\eta \colon \frac{\prod_{n=1}^\infty F_n}{\bigoplus_{n=1}^\infty F_n} \to \frac{\prod_{n=1}^\infty B_n}{\bigoplus_{n=1}^\infty B_n}
\end{equation}
be the cpc~order zero map induced by $(\eta_n)_{n=1}^\infty$. For a contraction $f\in
C_0((0,1])_+$,  the sequence $(f(\eta_n))_{n=1}^\infty$ induces the cpc order zero map
$f(\eta)$.\end{lemma}
\begin{proof}
Let $\rho_n\colon C_0((0,1])\otimes F_n\to B_n$ be the $^*$-homomorphism induced by $\eta_n$ as in
Proposition~\ref{prop.Cone}. We obtain a $^*$-homomorphism
\begin{equation}
\rho\colon C_0((0,1])\otimes \prod_{n=1}^\infty F_n/\bigoplus_{n=1}^\infty F_n\to\prod_{n=1}^\infty B_n/\bigoplus_{n=1}^\infty B_n,
\end{equation}
such that $\rho(f\otimes (x_n)_{n=1}^\infty) = (\rho_n(f\otimes x_n))_{n=1}^\infty$ for all $f \in
C_0((0,1])$ and $(x_n)^\infty_{n=1} \in \prod^\infty_{n=1} B_n / \bigoplus^\infty_{n=1} B_n$. So it suffices to check that $\rho$ is the $^*$-homomorphism associated to $\eta$ by Proposition~\ref{prop.Cone}. This follows as, for any $(x_n)_{n=1}^\infty  \in \prod_{n=1}^\infty B_n/\bigoplus_{n=1}^\infty B_n$, we have
\begin{equation}
\rho(\id_{(0,1]}\otimes (x_n))_{n=1}^\infty=(\eta_n(x_n))_{n=1}^\infty=\eta((x_n)_{n=1}^\infty).\qedhere
\end{equation}
\end{proof}

We next recall the definitions of nuclear dimension and decomposition rank,
from \cite{WZ:Adv} and \cite{KW:IJM} respectively. We work at the level of morphisms, these
definitions were first developed in \cite{TW:APDE}.

\begin{definition}\label{DefNDim}
Let $A$ and $B$ be $\Cstar$-algebras, let $\phi\colon A\to B$ be a $^*$-homo\-mor\-phism and let
$n\in\mathbb N$. Then $\phi$ is said to have \emph{nuclear dimension at most $n$} (written
$\dimnuc \phi\leq n$) if for any finite subset $\mathcal F \subseteq A$ and $\epsilon>0$, there
exist finite dimensional $\Cstar$-algebras $F^{(0)},\dots,F^{(n)}$ and maps
\begin{equation}
\xymatrix{
A \ar[r]^-{\psi} & F:=F^{(0)} \oplus \cdots \oplus F^{(n)} \ar[r]^-{\eta} & B
}
\end{equation}
such that $\psi$ is cpc, $\eta|_{F^{(i)}}$ is cpc\ order zero for $i=0,\dots,n$, and such that
\begin{equation}
\| \eta(\psi(x)) - \phi (x)\| \leq \epsilon,\quad x\in \mathcal F.
\end{equation}
Such an approximation $(F,\psi,\eta)$ is called an \emph{$n$-decomposable approximation} of $\phi$
for $\mathcal F$ up to $\epsilon$. So $\dimnuc \phi = \inf\{n : \dimnuc\phi \le n\}$ with the
convention that $\inf \emptyset = \infty$. If additionally the maps $\eta$ in these $n$-decomposable approximations can always be taken contractive, then $\phi$ is
said to have \emph{decomposition rank at most $n$}, written $\dr\,\phi\leq n$, and again
$\dr\,\phi = \inf\{n : \dr\,\phi\le n\}$. The \emph{nuclear dimension} and \emph{decomposition
rank} of a $\Cstar$-algebra $A$, written $\dimnuc A$ and $\dr\,A$ respectively, are defined to be
the nuclear dimension and decomposition rank of the identity map $\id_A$.
\end{definition}

Note that the definitions of nuclear dimension 0 and decomposition rank 0 coincide, but, for $n \ge
1$, that $\dr\,\phi = n$ is a stronger condition than that $\dimnuc\phi = n$.

The next result is a modification to maps of the fact that finite nuclear dimension and
decomposition rank are inherited by hereditary subalgebras (see \cite[Proposition~2.5]{WZ:Adv} and
\cite[Proposition~3.8]{KW:IJM} respectively).  The proof is a minor modification of
\cite[Proposition~2.5]{WZ:Adv} to this context.

\begin{proposition}\label{Prop:cpcHer}
Let $\phi \colon A \to B$ be a $^*$-homomorphism, and let $B_0 \subseteq B$ be a hereditary
$\Cstar$-subalgebra such that $\phi(A) \subseteq B_0$. Let $\phi_0 \colon A \to B_0$ be the
corestriction of $\phi$. Then $\dimnuc \phi = \dimnuc \phi_0$ and $\dr\,\phi=\dr\,\phi_0$.
\end{proposition}
\begin{proof}
Clearly $\dimnuc \phi_0 \geq \dimnuc \phi$ and $\dr\,\phi_0\geq\dr\,\phi$, so we will prove the
other inequalities. Assume first that $\dimnuc \phi \leq n < \infty$.

Let $a_1,\dots,a_m \in A$ be positive contractions, and $\epsilon > 0$. By perturbing each $a_j$
slightly, we may assume that there is a positive contraction $e \in A$, such that $e a_j = a_j$
for $j=1,\dots,m$. Let
\begin{equation}\label{eq:epsilon0min}
\epsilon_0 \coloneqq \min\left\{  \frac{\epsilon^8}{(35(n+1))^8}, \frac{1}{2^{18}} \right\}.
\end{equation}
 Let $h\in B_0$ be a positive contraction such that
\begin{equation}\label{eq:hphiepsilon}
\| h \phi(e) - \phi(e)\| < \epsilon_0,
\end{equation}
 and pick an $n$-decomposable cpc~approximation
\begin{equation}
\Big(F= \bigoplus_{i=0}^n F^{(i)},\ \psi \colon  A \to F,\ \eta=\sum_{i=0}^n\eta^{(i)}  \colon F \to B\Big)
\end{equation}
such that
\begin{equation}\label{eq:hereditary.neweq}
\|\eta(\psi(x))-\phi(x)\|<\epsilon_0,\quad x\in \{e, a_1,\dots, a_m\}.
\end{equation}
Let $\chi_{[\epsilon_0^{1/2} , 1]}$ denote the characteristic function of $[\epsilon_0^{1/2}, 1]$,
and let $p$ be the projection $p \coloneqq \chi_{[\epsilon_0^{1/2} , 1]} (\psi(e)) \in F$. Note
that
\begin{equation}\label{eq:pleqstuff}
p \leq \epsilon_0^{-1/2} \psi(e), \quad \textrm{and} \quad (1_F-p) \psi(e) \leq \epsilon_0^{1/2}1_F.
\end{equation}
For each $i=0,\dots,n$, define $p^{(i)} \coloneqq 1_{F^{(i)}} p$. Then (working in the unitisation $\widetilde{B}$ of $B$ for convenience),
\begin{align}
\|\eta^{(i)} (p^{(i)}) (1_{\widetilde{B}} - h)\|
    &= \| (1_{\widetilde{B}}-h) \eta^{(i)} (p^{(i)})^2 (1_{\widetilde{B}} - h)\|^{1/2}&& \nonumber\\
    &\leq   \| (1_{\widetilde{B}}-h) \eta^{(i)} (p^{(i)}) (1_{\widetilde{B}} - h)\|^{1/2}&&\nonumber\\
    &\leq \| (1_{\widetilde{B}}-h) \eta(p) (1_{\widetilde{B}}-h)\|^{1/2}&& \nonumber\\
    &\leq \epsilon_0^{-1/4} \| (1_{\widetilde{B}}-h) \eta(\psi(e)) (1_{\widetilde{B}}-h)\|^{1/2} \quad&&\text{by~\eqref{eq:pleqstuff}}\nonumber\\
    &\leq \epsilon_0^{-1/4} \| (1_{\widetilde{B}}-h) \eta(\psi(e)) \|^{1/2}&&\nonumber\\
    &\leq  \epsilon_0^{-1/4} (\| (1_{\widetilde{B}}-h) \phi(e)\| + \epsilon_0)^{1/2}\quad&&\text{by~\eqref{eq:hereditary.neweq}} \nonumber\\
    &\leq 2^{1/2} \epsilon_0^{1/4}\quad&&\text{by~\eqref{eq:hphiepsilon}} \label{eq:sqrt2epsilon}\\
    &\leq 1/16\quad&&\text{by~\eqref{eq:epsilon0min}.}&&\nonumber
\end{align}

Let $\widehat F \coloneqq p F p$ and $\widehat F^{(i)} \coloneqq pF^{(i)} p$ for $i=0,\ldots,n$,
and let $\widehat\psi \colon A \to \widehat{F}$ be the cpc map such that
\begin{equation}
\widehat \psi(a) = p \psi(a) p,\quad a \in A.
\end{equation}
By \cite[Lemma~3.6]{KW:IJM} and \eqref{eq:sqrt2epsilon}, for each $i \in \{0, \dots, n\}$ there
exists a cpc~order zero map
\begin{equation}
\hat \eta^{(i)} \colon \widehat F^{(i)} \to \overline{h B h} \subseteq B_0
\end{equation}
such that for all positive $x\in \widehat F^{(i)}$,
\begin{equation}
\| \widehat \eta^{(i)} (x) - \eta^{(i)} (x) \|
    \leq 8 \cdot 2^{1/4} \epsilon_0^{1/8}\| x\| 
    \leq 16 \epsilon_0^{1/8} \| x\|. \label{eq:16epsilon}
\end{equation}
Let
\begin{equation}
\widehat \eta  \coloneqq \sum_{i=0}^n \widehat \eta^{(i)} \colon \widehat F \to B_0,
\end{equation}
which is a sum of $n+1$ cpc~order zero maps by construction.

Fix $j \le m$. Since $\psi$ is cpc and each $a_j$ is a positive contraction, \eqref{eq:16epsilon}
gives
\begin{equation}\label{eq:hatvsnonhat}
\| \widehat \eta(\widehat \psi(a_j)) - \eta ( \widehat \psi(a_j)) \| \leq 16 (n+1)\epsilon_0^{1/8}.
\end{equation}
Using that $a_j \leq e$ by assumption at the third step, and using~\eqref{eq:pleqstuff} at the
final step, we calculate,
\begin{align}
\| (1_F - p) \psi(a_j) \|
    &= \| (1_F - p) \psi(a_j)^2 (1_F - p)\|^{1/2} \nonumber\\
    &\leq  \| (1_F - p) \psi(a_j) (1_F - p) \|^{1/2} \nonumber\\
    &\leq \| (1_F - p) \psi(e) (1_F - p) \|^{1/2} \nonumber\\
    &\leq \| (1_F - p) \psi(e) \|^{1/2} \nonumber\\
    &\leq \epsilon_0^{1/4}.\label{eq:1-ppsi}
\end{align}
Hence, using~\eqref{eq:1-ppsi} at the second step, we see that
\begin{align}
\| \widehat \psi(a_j) - \psi(a_j) \|
    &\leq \| p \psi(a_j) ( 1_F- p) \| + \| (1_F - p) \psi(a_j) \| \nonumber\\
    &\leq 2 \epsilon_0^{1/4}. \label{eq:hatnonhat}
\end{align}
Consequently, using \eqref{eq:hatvsnonhat}~and~\eqref{eq:hatnonhat} at the third step, and
then~\eqref{eq:epsilon0min} at the final step, we have
\begin{align}
\| \widehat \eta(\widehat \psi(a_j))&{} - \phi(a_j)\| \nonumber\\
    &\leq \| \widehat \eta(\widehat \psi(a_j)) - \eta(\widehat \psi(a_j)) \| + \| \eta(\widehat \psi(a_j) - \psi(a_j)) \| \nonumber\\
    &\qquad + \| \eta(\psi(a_j)) - \phi(a_j)\| \nonumber\\
    &\leq 16 (n+1)\epsilon_0^{1/8} + 2(n+1) \epsilon_0^{1/4} + \epsilon_0 \nonumber\\
    &\leq 19(n+1) \epsilon_0^{1/8} \label{eq:her,neweq2}\nonumber\\
    &\leq \epsilon,
\end{align}
and so $\dimnuc\phi_0\leq n$.

Finally note that when $\dr\,\phi\leq n$, we can additionally assume that $\|\eta\|\leq 1$. Then
\eqref{eq:16epsilon} ensures that $\|\widehat{\eta}\|\leq 1+16(n+1)\epsilon_0^{1/8}$.  Let
$\widetilde{\eta}\coloneqq\widehat{\eta}/\|\widehat{\eta}\|$. Then for $j=1,\dots,m$, using
\eqref{eq:her,neweq2}~and~\eqref{eq:epsilon0min} at the second and third steps respectively, we
have
\begin{align}
\|\widetilde{\eta}(\psi(a_j))-\phi(a_j)\|
    &\leq \|\widehat{\eta}(\psi(a_j))-\phi(a_j)\|+\|\widetilde{\eta}-\widehat{\eta}\|\nonumber\\
    &\leq 19(n+1)\epsilon_0^{1/8}+16(n+1)\epsilon_0^{1/8}\nonumber\\
    &\leq \epsilon,
\end{align}
and so $\dr\,\phi_0\leq n$.
\end{proof}

We end this section by discussing the difference between finite nuclear dimension and
decomposition rank, aiming for a criterion for recognising when a system of approximations which
a priori witnesses finite nuclear dimension also gives finite decomposition rank.  Firstly, if
$(F_i,\psi_i,\eta_i)_i$ is a net of $n$-decomposable approximations for $\id_A$ showing that $A$
has finite nuclear dimension, then \cite[Proposition~3.2]{WZ:Adv} shows that by removing certain
full matrix summands from the $F_i$ (and modifying $\psi_i$ and $\eta_i$ accordingly), one can
additionally assume that the maps $\psi_i\colon A\rightarrow F_i$ are approximately order zero. The
corresponding result for decomposition rank from \cite[Proposition~5.1]{KW:IJM} shows that if each
$\eta_i$ is also contractive, a suitable restriction of the approximations can be found so that
the maps $\psi_i\colon A\rightarrow F_i$ are approximately multiplicative; this is why finite
decomposition rank entails quasidiagonality.    Very similar proofs can be used to establish the
same results for any $^*$-homomorphism in place of the identity map.  We sketch the details of the
decomposition rank case, which we will use later in the paper.
\begin{proposition}\label{ApproxOrderZeroProp}
Let $\theta\colon A\rightarrow B$ be a $^*$-homomorphism between $\Cstar$-algebras with
$\dimnuc(\theta)\leq n$.  Then there exist a net of $n$-decomposable approximations
$(F_i,\psi_i,\eta_i)_i$ for $\theta$ such that the maps $\psi_i$ are approximately order zero,
i.e. the induced map $\psi\colon A\rightarrow\prod_iF_i/\bigoplus_iF_i$ is cpc order zero.  If additionally
$\dr\,\theta\leq n$, then the maps $\psi_i$ can be taken to be approximately multiplicative, i.e.
the induced map $\psi\colon A\rightarrow\prod_iF_i/\bigoplus_iF_i$ is a $^*$-homomorphism.
\end{proposition}
\begin{proof}
We will prove the second assertion, about $\theta$ with decomposition rank at most $n$. The
nuclear dimension statement follows from a very similar small modification to the corresponding
statement for algebras of finite nuclear dimension \cite[Proposition~3.2]{WZ:Adv}.\footnote{Note
that the second line of the proof of \cite[Proposition~3.2]{WZ:Adv} contains a typo; it should ask
for $0<\epsilon<\frac{1}{(n+2)^{16}}$.}

Fix a finite subset $\mathcal F\subseteq A$ of positive elements of norm $1$ and $0<\epsilon<1$.
Assume that $\theta$ has decomposition rank at most $n$.  Using the Stinespring calculation of
\cite[Lemma~3.4]{KW:IJM}, it suffices to find an $n$-decomposable approximation $(F,\psi,\eta)$ of
$\theta$ for $\mathcal F$ up to $\epsilon$ with $\eta$ contractive such that
\begin{equation}\label{ApproxOrderZeroProp.1}
\|\psi(x^2)-\psi(x)^2\|<\epsilon,\quad x\in\mathcal F.
\end{equation}
We closely follow the proof of \cite[Proposition~5.1]{KW:IJM}. Fix an $n$-decomposable
approximation $(\tilde F,\tilde \psi,\tilde \eta)$ of $\theta$ on $\mathcal F$ up to
$\frac{\epsilon^4}{6(n+1)}$ with $\tilde\eta$ contractive. Decompose $\tilde F=M_{r_1}\oplus\ldots\oplus
M_{r_s}$, and write $\tilde\psi_j$ and $\tilde\eta_j$ for the respective components of
$\tilde\psi$ and $\tilde\eta$. Set
\begin{equation}
I\coloneqq\{i\in\{1,\dots,s\}:\|\tilde\psi_i(x^2)-\tilde\psi_i(x)^2\|\geq\epsilon^2\text{ for some }x\in\mathcal F\}.
\end{equation}
Then, calculating identically\footnote{This is where the tolerance $\frac{\epsilon^4}{6(n+1)}$ plays a role in the proof.} to \cite[Proposition~5.1]{KW:IJM}, for $i\in I$, we have
$\|\sum_{i\in I}\tilde\eta_i(1_{M_{r_i}})\|\leq\frac{\epsilon^2}{2}$.  Now define $F \coloneqq
\bigoplus_{i\in\{1,\dots,s\}\setminus I}M_{r_i}$. Let $\psi\colon A\rightarrow F$ be the compression of
$\tilde\psi$ by $1_F$ to $F$ and let $\eta\colon F\rightarrow B$ be the restriction of $\tilde{\eta}$.
Note that~(\ref{ApproxOrderZeroProp.1}) holds by construction.  Moreover, for any contraction
$x\in A$, we have
\begin{equation}
\|\eta\psi(x)-\tilde\eta\tilde\psi(x)\|\leq \Big\|\sum_{i\in I}\tilde\eta_i(1_{M_{r_i}})\Big\|\leq\frac{\epsilon^2}{2}.
\end{equation}
Thus
\begin{equation}
\|\theta(x)-\eta\psi(x)\|\leq\|\theta(x)-\tilde\eta\tilde\psi(x)\|+\frac{\epsilon^2}{2}<\epsilon,\quad x\in \mathcal F.\qedhere
\end{equation}
\end{proof}

\begin{corollary}\label{DetectQDCor}
Let $\theta\colon A\rightarrow B$ be an injective $^*$-homomorphism with finite decomposition rank.
Then $A$ is quasidiagonal.
\end{corollary}
\begin{proof}
As quasidiagonality is a local property, we may assume that $A$ is separable (cf.
\cite[Exercise~1.1]{BO:Book}). Let $(F_n,\psi_n,\eta_n)_{n=1}^\infty$ be an approximation for
$\theta$ as in Proposition~\ref{ApproxOrderZeroProp}, so the map $\psi\colon A\rightarrow
\prod_{n=1}^\infty F_n/\bigoplus_{n=1}^\infty F_n$ induced by $(\psi_n)_{n=1}^\infty$ is a
$^*$-homomorphism.  The map $\eta\colon \prod_{n=1}^\infty F_n/\bigoplus_{n=1}^\infty F_n\rightarrow
B_{\seq}$ induced by $(\eta_n)_{n=1}^\infty$ satisfies $\iota\circ\theta=\eta\circ\psi$, where $\iota\colon B\rightarrow B_{\seq}$ is the canonical inclusion. As
$\iota\circ\theta$ is injective, so too is $\psi$, so $(\psi_n)_{n=1}^\infty$ witnesses
quasidiagonality of $A$.
\end{proof}

A kind of converse to Proposition~\ref{ApproxOrderZeroProp} provides a method for detecting when a
finite nuclear dimensional approximation $(F_i,\psi_i,\eta_i)$ also gives rise to finite
decomposition rank.  With the benefit of hindsight, we see that the criterion below is used to prove the
decomposition rank case of \cite[Theorem~F]{BBSTWW} and \cite[Theorem~B]{CETWW}, albeit in the case where $A$ is unital (when
Lemma~\ref{detectdr} has an easier proof).

\begin{lemma}\label{detectdr}
Let $\phi \colon A \to B$ be a $^*$-homomorphism and let $n\in \mathbb N_0$.  Suppose there is a
net $(F_i,\psi_i,\eta_i)_i$ of $n$-decomposable approximations converging point norm to $\phi$
such that the induced map
\begin{equation}
\Psi\coloneqq(\psi_i)\colon A\to\prod_i F_i/\bigoplus_i F_i,
\end{equation}
is a $^*$-homomorphism.  Then $\dr\,\phi\leq n$.
\end{lemma}
The point of the lemma is that the $\eta_i$ are not assumed to be contractive.

\begin{proof}
Let $(F_i,\psi_i,\eta_i)_i$ be as in the statement of the lemma, and fix a finite subset $\mathcal F\subseteq A$ of
contractions and $0 < \epsilon < 1$. Define $\delta \coloneqq (6n + 15)^{-1} \epsilon$. Pick
positive contractions $e_0,e_1\in A$ such that $e_0 e_1 = e_1$ and
\begin{equation}
\max\{\| a e_1 - a\|, \| e_1 a - a \|\} \leq \delta, \quad a\in \mathcal F.
\end{equation}

Let $g \colon [0,1] \to [0,1]$ denote the continuous function which is $0$ on
$[0,1-\delta]$, 1 on $[1-\delta/2, 1]$ and affine otherwise. By hypothesis $\Psi$ is a
$^*$-homomorphism, so as $e_0e_1=e_1$, we have
\begin{equation}
g (\Psi(e_0)) \Psi(e_1) = \Psi(e_1).
\end{equation}
Hence there exists $i$ such that
$(F_i,\psi_i,\eta_i)$ approximates $\phi$ on $\mathcal F\cup\{e_0,e_1\}$ up to $\delta$,
\begin{equation}\label{detectdr.neweq.1}
\psi_i(a)\approx_{4\delta}\psi_i(e_1)\psi_i(a)\psi_i(e_1),\quad a\in \mathcal F,
\end{equation}
and
\begin{equation}\label{eq:g1-delta2psim}
g (\psi_i(e_0)) \psi_i(e_1) \approx_\delta \psi_i(e_1).
\end{equation}
Let $\chi_{[1-\delta, 1]}$ denote the characteristic function of $[1-\delta, 1]$, and let $p_i
\coloneqq \chi_{[1-\delta, 1]}(\psi_i(e_0))\in F_i$ be the corresponding spectral projection.
Define $\widehat{F}_i\coloneqq p_iF_ip_i$, define $\widehat{\psi}_i\colon A\rightarrow \widehat{F}_i$ by
$\widehat{\psi}_i(\cdot) \coloneqq p_i\psi_i(\cdot)p_i$ and define
$\widehat{\eta}_i\colon \widehat{F}_i\rightarrow A$ by $\widehat{\eta}_i\coloneqq
\frac{1-\delta}{1+\delta}\eta_i|_{\widehat{F}_i}$. Note that $\widehat{\psi}_i$ is cpc and
$\widehat{\eta}_i$ is cp. The algebra $\widehat{F_i}$ inherits the decomposition into
$(n+1)$-summands on which $\widehat{\eta}_i$ restricts to a cpc order zero map from $\widehat{F}_i$ and
$\eta_i$.

By definition of $p_i$, we have $(1-\delta) p_i \leq \psi_i(e_0)$, so that
\begin{equation}
\eta_i((1-\delta) p_i ) \leq \eta_i(\psi_i(e_0)) \approx_\delta \phi(e_0).
\end{equation}
In particular $(1- \delta)\| \eta_i(p_i)\| \leq 1 + \delta$ and thus
\begin{equation}
\| \widehat{\eta}_i\| = \frac{1-\delta}{1+\delta}\| \eta_i(p_i)\| \leq 1.
\end{equation}
The definition of $p_i$ also ensures that
\begin{equation}\label{eq:g1-delta2}
g(\psi(e_0)) p_i = p_i g(\psi(e_0)) =  g(\psi(e_0)).
\end{equation}
For $a\in \mathcal F$, applying~\eqref{detectdr.neweq.1}, \eqref{eq:g1-delta2psim}, and
then~\eqref{eq:g1-delta2}, we have
\begin{align}
\widehat{\eta}_i(\widehat{\psi}_i(a))
    &\mathrel{\approx_{4\delta}} \tfrac{1-\delta}{1+\delta} \eta_i (p_i \psi_i(e_1) \psi_i(a) \psi_i(e_1) p_i) \nonumber\\
    &\approx_{2\delta} \tfrac{1-\delta}{1+\delta} \eta_i (p_i g(\psi_i(e_0)) \psi_i(e_1) \psi_i(a) \psi_i(e_1) g(\psi_i(e_0)) p_i)\nonumber \\
    &= \tfrac{1-\delta}{1+\delta} \eta_i (g(\psi_i(e_0)) \psi_i(e_1) \psi_i(a) \psi_i(e_1) g(\psi(e_0))).
\end{align}
Now, using that $\|\eta_i\|\leq n+1$, we continue the calculation above, obtaining
\begin{align}
\widehat{\eta_i}(\widehat{\psi}_i(a))
&\approx_{6\delta+2(n+1)\delta} \tfrac{1-\delta}{1+\delta} \eta_i (\psi_i(e_1) \psi_i(a) \psi_i(e_1))\quad&&\text{by~\eqref{eq:g1-delta2psim}} \nonumber\\
&\approx_{4(n+1)\delta} \tfrac{1-\delta}{1+\delta} \eta_i(\psi_i(a)))\quad&&\text{by~\eqref{detectdr.neweq.1}} \nonumber\\
&\approx_{\delta} \tfrac{1-\delta}{1+\delta} \phi(a).	
\end{align}
Since $1-\tfrac{1-\delta}{1+\delta}\leq 2\delta$, this gives
\begin{equation}
\| \phi(a) - \widehat{\eta}_i \widehat{\psi}_i(a)\| \leq 9 \delta + 6(n+1)\delta  = \epsilon,\quad a\in \mathcal F.
\end{equation}
Hence $\dr\,\phi\leq n$.
\end{proof}


\section{The ideal lattices of $\Cstar$-algebras}\label{S2}

In this section, we collect various facts regarding the ideal lattice of a $\Cstar$-algebra, which
we need for our main existence argument.

Here, and throughout, by an ideal in a $\Cstar$-algebra $A$
we always mean a closed, two-sided ideal. As is customary, $\overline{A S A}$ denotes the ideal generated by a non-empty subset $S \subseteq
A$. So $\overline{ASA} = \overline{\operatorname{span}}\{axb : a,b \in A,\ x \in S\}$. If $S=\{a\}$
is a singleton, we write $\overline{AaA}$ rather than $\overline{A\{a\}A}$.

Given an ideal $I$ in a $\Cstar$-algebra $A$, an element $a\in I$ is \emph{full}, if it generates
$I$ as an ideal; that is, if $I=\overline{AaA}$.  Every ideal in a separable $\Cstar$-algebra has
a full element (for example, any strictly positive element is full).

The lattice $\mathcal I(A)$ of ideals of a $\Cstar$-algebra $A$ is complete in the sense that
every $\mathcal S \subseteq \mathcal I(A)$ has a supremum and an infimum. Indeed,
\begin{equation}
\sup \mathcal S = \overline{\sum_{I \in \mathcal S} I}, \qquad\text{ and }\qquad \inf \mathcal S = \bigcap_{I\in \mathcal S} I
\end{equation}
for every non-empty subset $\mathcal S \subseteq \mathcal I(A)$; by convention, we take $\sup
\emptyset = 0$ and $\inf \emptyset = A$. There is a notion of compact containment in $\mathcal
I(A)$: if $I,J \in \mathcal I(A)$, then $I$ is \emph{compactly contained} in $J$, written $I
\Subset J$, if whenever $(I_\lambda)_{\lambda\in \Lambda}$ is a family in $\mathcal I(A)$ such
that $J \subseteq \overline{\sum_{\lambda \in \Lambda} I_\lambda}$ then there are finitely many
$\lambda_1,\dots, \lambda_n \in \Lambda$ such that $I \subseteq \sum_{i=1}^n I_{\lambda_i}$.

With suprema and compact containment, the ideal lattice of a separable $\Cstar$-algebra fits into
the framework of the abstract Cuntz semigroup category $\Cu$ introduced in \cite{CEI08} (see
\cite[Proposition~2.5]{Gabe17}). Given $\Cstar$-algebras $A$ and $B$, a map between their ideal
lattices $\Phi \colon \mathcal I(A) \to \mathcal I(B)$ is a \emph{$\Cu$-morphism} if it preserves
suprema and compact containment. By \cite[Lemma~2.12]{Gabe17} any $^*$-homomorphism $\phi \colon A
\to B$ induces a $\Cu$-morphism
\begin{equation}
\mathcal I(\phi) \colon \mathcal I(A) \to \mathcal I(B), \qquad I \mapsto \overline{B \phi(I) B}.
\end{equation}
Then $\mathcal I(\cdot)$ is a covariant functor from $\Cstar$-algebras and $^\ast$-homomorphisms
to complete lattices and $\Cu$-morphisms \cite[Proposition~2.15]{Gabe17}.

The presence of a full element in an ideal is detected by the ideal lattice structure
(\cite[Corollary~2.3]{Gabe17}), and so preserved by $\Cu$-morphisms.

\begin{lemma}\label{lem:fullelement}
Let $A,B$ be $\Cstar$-algebras with $A$ separable, and let $\Theta\colon \mathcal I(A)\rightarrow
\mathcal I(B)$ be a $\Cu$-morphism.  Then, for every non-zero ideal $J\in \mathcal I(A)$, the
ideal $\Theta(J)$ has a full element.
\end{lemma}
\begin{proof}
Since $A$ is separable, $J$ has a full element.  So \cite[Corollary~2.3]{Gabe17} gives
$J=\overline{\bigcup_{n=1}^\infty J_n}$ for some sequence of ideals $J_n \in \mathcal{I}(A)$ such
that $J_n \Subset J_{n+1}$ for all $n$. As $\Theta$ preserves suprema and compact containment,
$\Theta(J)=\overline{\bigcup_{n=1}^\infty\Theta(J_n)}$, and $\Theta(J_n)\Subset\Theta(J_{n+1})$
for all $n$. Hence $\Theta(J)$ has a full element by \cite[Corollary~2.3]{Gabe17}.
\end{proof}

It is well-known that the ideal lattice of a separable $\Cstar$-algebra is countably generated.
The next lemma provides a slight strengthening of this statement.

\begin{lemma}\label{l:basis}
Let $A$ be a separable $\Cstar$-algebra. Then there is a sequence $(J_n)_{n=1}^\infty$ in
$\mathcal I(A)$, such that for any $J\in \mathcal I(A)$ there is a strictly increasing sequence
$(k_n)^\infty_{n=1}$ in $\mathbb N$, such that $J_{k_n} \Subset J_{k_{n+1}}$ for all $n$, and
$\overline{\bigcup_{n=1}^\infty J_{k_n}} = J$.
\end{lemma}
\begin{proof}
As $A$ is separable, \cite[Corollary~4.3.4]{P:Book} shows that $\mathcal I(A)$ has a countable
basis $\mathcal B \subseteq \mathcal I(A)$; that is, for every $I\in\mathcal I(A)$, there exists
$\mathcal B_I\subseteq \mathcal B$ such that $I=\sup\mathcal B_I$. Replacing $\mathcal B$ with
$\{\sum_{I \in \mathcal S} I : \emptyset \not= \mathcal S \subseteq \mathcal B\text{ is
finite}\}$, we may assume that $\mathcal B$ is upwards directed.  Let $(J_n)_{n=1}^\infty$ be a
sequence in $\mathcal B$ in which each $I \in \mathcal B$ appears infinitely often; that is, $\{n
\in \mathbb N : J_n = I\}$ is infinite for each $I\in \mathcal B$.

Fix $J \in \mathcal I(A)$. As $A$ is separable, $J$ contains a full element, so by
\cite[Corollary~2.3]{Gabe17} there is a sequence $I_1 \Subset I_2 \Subset \dots$ in $\mathcal
I(A)$ such that $J = \overline{\bigcup_{n=1}^\infty I_n}$. As $\mathcal B$ is an upwards directed
basis, we can use the compact containment relation to find $M_1\in\mathcal B$ such that $I_{2}
\subseteq M_1 \subseteq I_{3}$. In a similar manner, for each $n$ we can find $M_n\in\mathcal B$
such that $I_{2n} \subseteq M_n \subseteq I_{2n+1}$ for each $n\in\mathbb N$. Since each element
in $\mathcal B$ appears infinitely often in $(J_n)_{n=1}^\infty$, we can find a sequence $k_1 <
k_2 < \dots$ such that $M_n=J_{k_n}$. This together with the compact containment gives
$I_2\subseteq J_{k_1}\subseteq I_3\Subset I_4\subseteq J_{k_2} \subseteq I_5\Subset \dots$,
consequently $J_{k_1}\Subset J_{k_2}\Subset\dots$, and $J = \overline{\bigcup_{n=1}^\infty
J_{k_n}}$.
\end{proof}

The previous lemma allows us to approximate the ideal lattice of a separable $\Cstar$-algebra by
an increasing sequence of finite lattices.  The definition and lemma below collect some notation
and an easy fact regarding these finite lattices.

\begin{definition}\label{DefUPD}
Let $\mathcal I$ be a finite lattice. Given $I\in\mathcal I$, we call $J\in\mathcal I$ a
\emph{predecessor} of $I$ if $J<I$ and $\{K\in\mathcal I : J\leq K\leq I\} = \{J, I\}$.  Note that an element
$J\in\mathcal I$ is the unique predecessor of $I\in\mathcal I$ if $\{K \in \mathcal I : K < I\} =
\{K \in \mathcal I : K \le J\}$. We write $\mathcal{UP}(\mathcal I)$ for the set of elements in
$\mathcal I$ that have a unique predecessor. For $I \in \mathcal{UP} (\mathcal I)$, we write $P(I)$
for the unique predecessor of $I$.
\end{definition}

For the next lemma, note that every finite lattice $\mathcal I$ has a minimum element $
\inf_{I \in \mathcal I} I$, which is by convention the supremum in $\mathcal I$ of the empty set.

\begin{lemma}\label{lem:unionpredecessors}
	Let $\mathcal I$ be a finite lattice. Then any $I\in\mathcal I$ satisfies
	\begin{equation}\label{UnionPredecessors.1}
	I = \sup_{\substack{J \in \mathcal{UP}(\mathcal I) \\ J \leq I}} J.
	\end{equation}
\end{lemma}
\begin{proof}
	As the empty supremum defines the minimal element, (\ref{UnionPredecessors.1}) holds for the minimal element of $\mathcal I$.  Fix $I\in\mathcal I$ and suppose inductively that
(\ref{UnionPredecessors.1}) holds for all $J\in\mathcal I$ with $J\leq I$ and $J\neq I$.  If $I$
has a unique predecessor then (\ref{UnionPredecessors.1}) holds vacuously.  Otherwise $I$ is the
join of its predecessors, each of which satisfies (\ref{UnionPredecessors.1}), so $I$ satisfies
(\ref{UnionPredecessors.1}).
\end{proof}

For nested ideals $J \subseteq I$ in a $\Cstar$-algebra $A$, let
\begin{equation}\label{eq:M_A(I,J)}
M_A(I,J) \coloneqq  \{ x \in A : x I \cup Ix \subseteq J \}.
\end{equation}
This is readily seen to be an ideal in $A$.  The notation is chosen, as $M_A(I,J)$ consists of those elements of $A$ which multiply $I$ into $J$.

The next lemma combines the above construction of ideals with the unique predecessors.

\begin{lemma}\label{l:JinKPK}
Let $A$ be a $\Cstar$-algebra, and let $\mathcal I \subseteq \mathcal I(A)$ be a finite
sublattice. Suppose that $J\in \mathcal I$ and $K \in \mathcal{UP}(\mathcal I)$ satisfy $K \not
\subseteq J$. Then $J \subseteq M_A(K,P(K))$.
\end{lemma}
\begin{proof}
Since $K \not \subseteq J$, we have $J \cap K \subsetneq K$. Since $P(K)$ is the unique
predecessor of $K$ in $\mathcal I$, we have $J\cap K \subseteq P(K)$. Since $J \cap K = KJ = JK$,
we deduce that $J \subseteq M_A(K,P(K))$.
\end{proof}

The following `lower bound' lemma plays a key role in our proof of Sublemma~\ref{sl:2} below, which we use to  control our `upward maps'. In the following proof and elsewhere,
given a positive element $a$ of a $\Cstar$-algebra and a positive constant $\lambda$, we write $(a
- \lambda)_+$ for the element obtained by applying functional calculus for $a$ to the function $t
\mapsto \max\{t - \lambda, 0\}$.

\begin{lemma}\label{l:alowerbound}
Let $A$ be a $\Cstar$-algebra and let $J\Subset I$ be a compact containment of ideals in $A$.  For
any full, positive element $a\in I$, there exists a constant $\lambda>0$ such that, for any ideals
$K_1 \subsetneq K_2\subseteq J$, we have
\begin{equation}
\|a + M_A(K_2,K_1)\|_{A/M_A(K_2,K_1)}\geq\lambda.
\end{equation}
\end{lemma}
\begin{proof}
As $I=\overline{\bigcup_{\lambda>0}A(a-\lambda)_+A}$, the compact containment $J\Subset I$
provides $\lambda>0$ such that  $J \subseteq \overline{A (a-\lambda)_+A}$.

Fix ideals $K_1 \subsetneq K_2 \subseteq J$. Suppose for contradiction that $(a-\lambda)_+ \in
M_A(K_2,K_1)$. Then
\begin{equation}
K_2 \subseteq \overline{A (a-\lambda)_+A } \subseteq M_A(K_2,K_1).
\end{equation}
Then $K_2\subseteq K_2\cap M_A(K_2,K_1)=K_2\cdot M_A(K_2,K_1)\subseteq K_1$, by definition of $M_A(K_2,K_1)$,
contradicting $K_1 \subsetneq K_2$. Hence $(a-\lambda)_+ \notin M_A(K_2,K_1)$, and so
\begin{equation}
\| a + M_A(K_2,K_1) \|_{A/M_A(K_2,K_1)} \geq \lambda. \qedhere
\end{equation}
\end{proof}

Before turning to our main construction, we record a standard fact regarding the behaviour of
ideals in $C_0((0,1])$ under the multiplication map.

\begin{lemma}\label{l:mcomm}
Let $m \colon [0,1] \times [0,1] \to [0,1]$ be the multiplication map, and $m^\ast \colon
C_0((0,1]) \to C_0((0,1]) \otimes C_0((0,1])$ be the induced $^*$-homomorphism. For every
non-zero, positive element $h\in C_0((0,1])$ there are positive, non-zero contractions $f,g \in
C_0((0,1])$ such that $g(1) = 1$ and $f \otimes g \leq m^\ast(h)$.

In particular, for any non-zero $I \in \mathcal I(C_0((0,1]))$, the ideal $\mathcal I(m^\ast)(I)$
contains a non-zero, positive contraction of the form $f \otimes g$, for which $g(1) = 1$.
\end{lemma}
\begin{proof}
We may find $\delta >0$ and a non-empty, open subset $U \subseteq (0,1]$ such that $h(s) \geq
\delta$ for $s \in U$. Fix $s\in U \setminus \{1\}$. As $(s,1) \in m^{-1}(U)$ we may pick
$\epsilon > 0$ such that $[s-\epsilon, s+\epsilon] \times [1-\epsilon, 1] \subseteq m^{-1}(U)$.
Let $f, g \in C_0((0,1])$ be non-zero, positive contractions supported in $[s-\epsilon,
s+\epsilon]$ and $[1-\epsilon, 1]$ respectively, and such that $\| f\| \leq \delta$ and $g (1) =
1$. Then $f\otimes g \leq m^\ast(h)$.
\end{proof}




\section{Main technical construction}\label{S3}

In this section we provide the main technical construction of the paper, obtaining
$^*$-homomorphisms $C_0((0,1])\otimes A \to B_{\seq}$ induced by sequences of maps of nuclear
dimension zero with specified behaviour on ideals.  Our first lemma is the main tool for producing
the `downward' maps, using Voiculescu's quasidiagonality of cones \cite{Voi91} in the spirit of
earlier nuclear dimension computations \cite{MS:DMJ,SWW:Invent,BBSTWW}.

\begin{lemma}\label{l:conetomatrix}
Let $A$ be a separable $\Cstar$-algebra, and let $J \subsetneq I \subseteq A$ be ideals. Then
there exist a sequence $(M_{k_n})_{n=1}^\infty$ of matrix algebras and a sequence
$(\psi_n)^\infty_{n=1}$ of cpc maps $\psi_n \colon C_0((0,1]) \otimes A \to M_{k_n}$ such that the
induced map
\begin{equation}
\psi \coloneqq (\psi_n)_{n=1}^\infty  \colon C_0((0,1])\otimes A\rightarrow\prod_n M_{k_n}/\bigoplus_n M_{k_n}
\end{equation}
is a $^*$-homomorphism such that for all $f\in C_0((0,1])$ and $a\in A$,
\begin{equation}\label{l:conetomatrix.1}
|f(1)| \, \| a + M_A(I,J)\|_{A/M_A(I,J)} \leq \|\psi(f \otimes a)\| \leq \| f\| \, \|a + M_A(I,J)\|_{A/M_A(I,J)}.
\end{equation}
If $A/M_A(I,J)$ is quasidiagonal we can choose the $\psi_n$ so that $\psi_n(C_0((0,1)) \otimes A) = 0$
for all $n\in \mathbb N$.
\end{lemma}
\begin{proof}
By Voiculescu's quasidiagonality of cones \cite{Voi91} there exist integers $k_n$ and cpc maps
\begin{equation}
\widetilde{\psi}_n \colon C_0((0,1]) \otimes (A/M_A(I,J)) \to M_{k_n}
\end{equation}
such that the map $\widetilde \psi \colon C_0((0,1])\otimes A/M_A(I,J) \to \prod_n M_{k_n}/\bigoplus_n
M_{k_n}$ induced by the $\tilde\psi_n$ is an injective $^*$-homomorphism. Let $q_{M_A(I,J)} \colon A \to A/M_A(I,J)$
denote the quotient map, and for each $n \in \mathbb N$, let
\begin{equation}
\psi_n \coloneqq  \widetilde{\psi}_n \circ (\id_{C_0((0,1])} \otimes q_{M_A(I,J)}) \colon C_0((0,1]) \otimes A \to M_{k_n}.
\end{equation}
Let $\psi \colon C_0((0,1]) \otimes A \to \prod_{n=1}^\infty M_{k_n} / \bigoplus_{n=1}^\infty M_{k_n}$ be
the induced map, so $\psi = \widetilde{\psi} \circ (\id_{C_0((0,1])} \otimes q_{M_A(I,J)})$. Then for
$f \in C_0((0,1])$ and $a \in A$, since $\widetilde{\psi}$ is isometric,
\begin{equation}
\begin{split}
\| \psi(f\otimes a) \| &= \| \widetilde \psi(f \otimes q_{M_A(I,J)}(a))\|\\
    &= \| f\| \, \|q_{M_A(I,J)}(a)\| \geq |f(1)| \, \|q_{M_A(I,J)}(a)\|.
\end{split}
\end{equation}

If $A/M_A(I,J)$ is quasidiagonal, then there is a sequence of integers $k_n$ and there are cpc
maps $\overline \psi_n \colon A/M_A(I,J) \to M_{k_n}$ such that the induced map $\overline \psi
\coloneqq (\overline \psi_n)_{n=1}^\infty \colon A/M_A(I,J) \to \prod_n M_{k_n}/\bigoplus_n
M_{k_n}$ is an injective $^*$-homomorphism. Let $\mathrm{ev}_1 \colon C_0((0,1]) \otimes A \to A$
be evaluation at $1$, and for each $n$, let
\begin{equation}
\psi_n \coloneqq  \overline{\psi}_n \circ q_{M_A(I,J)} \circ \mathrm{ev}_1 \colon C_0((0,1]) \otimes A \to M_{k_n},
\end{equation}
so that $\psi_n(C_0((0,1))\otimes A)=0$.  Let $\psi \colon A \to \prod_{n=1}^\infty M_{k_n} /
\bigoplus_{n=1}^\infty M_{k_n}$ be the map induced by $(\psi_n)_{n=1}^\infty$. Then
\begin{equation}
\begin{split}
\| \psi(f \otimes a) \| &= |f(1)| \, \| \overline \psi(q_{M_A(I,J)}(a))\|\\
    &= | f(1) | \, \|q_{M_A(I,J)}(a)\| \leq \| f\|\, \|q_{M_A(I,J)}(a)\|
\end{split}
\end{equation}
for all $a\in A$ and $f\in C_0((0,1])$.
\end{proof}

Recall that if $a,b\in A$ are positive elements, then $a$ is said to be \emph{Cuntz dominated} by
$b$, written $a \precsim b$, if there is a sequence of elements $(z_n)_{n=1}^\infty$ in $A$ such
that $z_n^\ast b z_n \to a$. In general it is not possible to choose $(z_n)_{n=1}^\infty$ to be
bounded without additional assumptions. The next lemma identifies one such assumption.

\begin{lemma}\label{l:Cuntzdompi}
Let $A$ be a $\Cstar$-algebra, and let $a,b\in A_+$. Suppose that for every continuous function $f
\colon [0,\| b \|] \to [0,1]$ with $f(0) = 0$ and $f(\| b \|) \neq 0$, we have  $a\precsim f(b)$.
Then there is a sequence $(z_n)_{n=1}^\infty$ in $A$ such that $z_n^\ast b z_n \to a$ and
\begin{equation}
\| z_n \| = (\| a\| / \| b\|)^{1/2},\quad n\in \mathbb N.
\end{equation}
\end{lemma}
\begin{proof}
We may assume that $\| a \| = \| b\| =1$. For each $n\in\mathbb N$, let $\delta_n \coloneqq  1 -
1/n$. Our hypothesis implies that for each $n$ we have $a \preceq (b - \delta_n)_+$, so there
exists $y_n \in A$ such that $\| y_n^\ast (b - \delta_n)_+ y_n - a\| < 1/n$. Let $z_n'\coloneqq
(b-\delta_n)_+^{1/2} y_n$ so that $\|z_n'\|\rightarrow \|a\|^{1/2} = 1$, and let $z_n \coloneqq
z'_n/\|z'_n\|$ for each $n$. Then each $\|z_n\|=1$ and $z_n^*z_n\rightarrow a$.

Let $f_n \colon [0,1] \to [0,1]$ be the continuous function which is $0$ on $[0, \delta_n - 1/n]$,
$1$ on $[\delta_n , 1]$ and affine otherwise. Then $f_n(b) z_n = z_n$, and
\begin{equation}
\|f_n(b) - f_n(b)b\| \leq \max\big\{ \sup_{t\in [\delta_n,1]} (1-t), \sup_{t\in [\delta_n -1/n, \delta_n]} f_n(t) (1-t)\big\} \leq 2/n.
\end{equation}
Hence
\begin{equation}
z_n^\ast b z_n = z_n^\ast f_n(b) b z_n \approx_{2/n} z_n^\ast f_n(b) z_n = z_n^\ast z_n.
\end{equation}
Therefore $z_n^\ast b z_n \to a$.
\end{proof}

We now turn to the construction of the `upwards maps'. Here we need the additional space given by
absorption of the Cuntz algebra $\mathcal O_\infty$ from \cite{C:CMP}. Recall that a
$\Cstar$-algebra $B$ is $\mathcal O_\infty$-stable if $B \otimes \mathcal O_\infty \cong B$. By
\cite[Proposition~4.5]{KR-AJM}, each such $\Cstar$-algebra $B$ is \emph{purely infinite} in the
sense that $B$ has no characters and for all $a,b\in B_+$ with $a \in \overline{BbB}$, we
have $a \precsim b$. We begin by recording a by-now-standard consequence of $\mathcal
O_\infty$-stability for repeated later use (and also as results of this type in the literature are
often stated with separability hypotheses).
\begin{lemma}\label{lem:kappa}
Let $B$ be a $\Cstar$-algebra such that $B\cong B\otimes\mathcal O_\infty$.  Then there exists an
isomorphism $\kappa\colon B\otimes\mathcal O_\infty\rightarrow B$ such that $\mathcal
I(\kappa)(K\otimes\mathcal O_\infty)=K$ for all $K\in\mathcal I(B)$.
\end{lemma}
\begin{proof}
Fix any isomorphism $\theta\colon B\cong B\otimes\mathcal O_\infty$. There exists an isomorphism
$\alpha\colon \mathcal O_\infty\rightarrow\mathcal O_\infty\otimes\mathcal O_\infty$ which is
approximately unitarily equivalent to the first factor embedding $\id_{\mathcal O_\infty}\otimes
1_{\mathcal O_\infty}\colon \mathcal O_\infty\rightarrow\mathcal O_\infty\otimes\mathcal
O_\infty$.\footnote{By \cite[Theorem~3.15]{KP:crelle}, one has $\mathcal O_\infty \cong \mathcal
O_\infty \otimes \mathcal O_\infty$, and by \cite[Theorem~3.3]{LP:crelle}, any two unital
$^*$-homomorphisms $\phi, \psi \colon \mathcal O_\infty \to \mathcal O_\infty \otimes \mathcal
O_\infty$ are approximately unitarily equivalent.}  Define $\kappa$ to be the composition
\begin{equation}
B\otimes\mathcal O_\infty\stackrel{\theta\otimes\id_{\mathcal O_\infty}}{\longrightarrow}B\otimes\mathcal O_\infty\otimes\mathcal O_\infty\stackrel{\id_B\otimes\alpha^{-1}}\longrightarrow B\otimes\mathcal O_\infty\stackrel{\theta^{-1}}\longrightarrow B.
\end{equation}
By construction $\kappa$ is an isomorphism and $\kappa^{-1}$ is approximately unitarily equivalent
(via unitaries in $B^\sim\otimes\mathcal O_\infty$) to the first factor embedding $\id_B\otimes
1_{\mathcal O_\infty}\colon B\rightarrow B\otimes\mathcal O_\infty$. Therefore $\mathcal
I(\kappa^{-1})=\mathcal I(\id_B\otimes 1_{\mathcal O_\infty})$.  Since $\mathcal I$ is functorial
(\cite[Proposition~2.15]{Gabe17}), $\mathcal{I}(\kappa)=\mathcal I(\id_B\otimes 1_{\mathcal
O_\infty})^{-1}$, so that $\mathcal I(\kappa)(K\otimes\mathcal O_\infty)=K$ for all $K\in\mathcal
I(B)$.
\end{proof}

The next lemma produces appropriate families of pairwise orthogonal positive elements $(h_I)_{I\in\mathcal{UP}(I)}$. We
shall produce our `upwards maps' out of finite dimensional algebras by embedding the finite
dimensional algebras into $\mathcal O_\infty$ and tensoring by these $(h_I)$. The symbol $\mathcal{UP}(\mathcal I)$ is defined in Definition~\ref{DefUPD}.

\begin{lemma}\label{l:upwardsmap}
Let $A$ and $B$ be $\Cstar$-algebras such that $A$ is separable and $B \otimes \mathcal O_\infty
\cong B$. Suppose that $\Theta \colon \mathcal I(A) \to \mathcal I(B)$ is a $\Cu$-morphism. Let
$\mathcal I_0 \subseteq \mathcal I(A)$ be a finite sublattice. Then there exists a collection
$\{h_I: I \in \mathcal{UP}(\mathcal I_0)\} \subseteq B$ of pairwise orthogonal positive elements,
each with spectrum $[0,1]$ such that
\begin{itemize}
\item[$(a)$] for each $I \in \mathcal{UP}(\mathcal I_0)$, the element $h_I$ is full in
    $\Theta(I)$, and
\item[$(b)$] for any pair of ideals $J_1 \subseteq J_2$ in $\mathcal I_0$, such that $J_1
    \Subset J_2$ in $\mathcal I(A)$, and any non-zero, continuous function $f\colon [0,1] \to
    [0,1]$ with $f(0) = 0$, $\Theta(J_1)$ is contained in the ideal of $B$ generated by
\begin{equation}
\sum_{\substack{I \in \mathcal{UP}(\mathcal I_0) \\ I \subseteq J_2}} f(h_I).
\end{equation}
\end{itemize}
\end{lemma}
\begin{proof}
By Lemma~\ref{lem:fullelement}, for each non-zero $I\in\mathcal{UP}(\mathcal I_0)$, the ideal
$\Theta(I)$ contains a full positive element $k_I$ of norm $1$. Use Lemma~\ref{lem:kappa} to
fix an isomorphism $\kappa\colon B\otimes\mathcal O_\infty\rightarrow B$ such that $\mathcal
I(\kappa)(K\otimes\mathcal O_\infty)=K$ for all $K\in\mathcal I(B)$.  Choose non-zero pairwise
orthogonal projections $(p_I)_{I\in \mathcal I_0}$ in $\mathcal O_\infty$, and define pairwise
orthogonal positive elements of norm $1$ by $h_I'\coloneqq \kappa(k_I\otimes p_I)$ for
$I\in\mathcal{UP}(\mathcal I_0)$.  Since $\mathcal O_\infty$ is simple and each $k_I$ is full in
$\Theta(I)$, each $k_I \otimes p_I$ is full in $\Theta(I) \otimes \mathcal O_\infty$. Hence, by the definition of
$\kappa$, each $h_I'$ is full in $\Theta(I)$.

For $n\in\mathbb N$, let $g_n \colon [0,1] \to [0,1]$ be the continuous function which is $0$ on
$[0,\frac{1}{2n}]$, $1$ on $[\frac{1}{n}, 1]$ and affine on $[\frac{1}{2n}, \frac{1}{n}]$. Fix
$J_1 \subseteq J_2$ in $\mathcal I_0$ such that $J_1 \Subset J_2$ in $\mathcal I(A)$.
Lemma~\ref{lem:unionpredecessors} implies that
\begin{equation}
J_2=\sum_{\substack{I\in\mathcal{UP}(\mathcal I_0) \\ I\subseteq J_2}}I.
\end{equation}
Since ideals are hereditary, for positive elements $a_1, \dots, a_n$ of a $\Cstar$-algebra $D$, we
have $\overline{D \big(\sum_i a_i\big)D} = \sum_i \overline{D a_i D}$. Thus the positive element
\begin{equation}\label{eq:l:upwardsmap.neweq.2}
h_{J_2}'' \coloneqq \sum_{\substack{I \in \mathcal{UP}(\mathcal I_0) \\ I \subseteq J_2}} h_I'
\end{equation}
is full in $\Theta(J_2)$. As $(g_n(h_{J_2}''))_{n=1}^\infty$ is an increasing approximate identity
for $\mCstar(h_{J_2}'')$, it follows that
\begin{equation}
\overline{\bigcup_{n=1}^\infty \overline{B g_n(h_{J_2}'') B}} = \overline{ B h_{J_2}'' B} = \Theta(J_2).
\end{equation}
As $\Theta$ is a $\Cu$-morphism, and $J_1 \Subset J_2$, it follows that $\Theta(J_1) \Subset
\Theta(J_2)$. Hence there exists $n_{J_1,J_2} \in \mathbb N$, such that every $n \ge n_{J_1, J_2}$
satisfies
\begin{equation}\label{eq:J1ingdelta}
\Theta(J_1) \subseteq \overline{B g_n(h_{J_2}'') B}.
\end{equation}

Since $\mathcal I_0$ is finite, we may define $n_0 \coloneqq  \max\{n_{J_1, J_2} : J_1, J_2 \in
\mathcal{I}_0\text{ and } J_1 \Subset J_2\text{ in }\mathcal{I}(A)\}$. Let $r\colon [0,1]\rightarrow
\mathbb [0,1]$ be the continuous function satisfying $r(t)=(2n_0)t$ for $t\in [0,\frac{1}{2n_0}]$
and $r(t)=1$ for $t\geq \frac{1}{2n_0}$. So $r(t)g_{n_0}(t)=g_{n_0}(t)$ for all $t$. Fix $h \in
(\mathcal O_\infty)_+$ with spectrum $\sigma(h) = [0,1]$, and for each $I\in \mathcal{UP}(\mathcal
I_0)$, let $h_I \coloneqq  \kappa(r(h_I') \otimes h)$. Each $h_I$ is positive with spectrum
$[0,1]$ since both $r(h_I')$ and $h$ are positive elements with spectrum $[0,1]$ (see \cite{BP66}).

By choice of $\kappa$ and by simplicity of $\mathcal O_\infty$, we have $\overline{B h_I B} =
\overline{B r(h_I') B}$ for each $I\in \mathcal{UP}(\mathcal I_0)$. For each $I\in\mathcal
P(\mathcal I_0$), the element $r(h_I')$ is a positive contraction in $\Theta(I)$ with $r(h_I')\geq
h_I'$, and so $r(h_I')$ has norm one and is full in $\Theta(I)$. Hence $h_I$ is full in
$\Theta(I)$. Since the $h_I'$ are pairwise orthogonal positive contractions so are the $h_I$. It
remains to check~$(b)$.

Take $J_1 \subseteq J_2$ in $\mathcal I_0$ with $J_1 \Subset J_2$ in $\mathcal I(A)$. Let $f\colon
[0,1] \to [0,1]$ be non-zero and continuous with $f(0) = 0$. Let $K$ be the ideal generated by
\begin{equation}
\sum_{\substack{I \in \mathcal{UP}(\mathcal I_0)\\ I \subseteq J_2}} f(h_I).
\end{equation}
We must show that $\Theta(J_1) \subseteq K$.

Let $\Phi_I \colon C_0((0,1]) \otimes C_0((0,1]) \to B \otimes \mathcal O_\infty$ be the $^\ast$-homomorphism given on elementary tensors by $\Phi_I(f_1 \otimes f_2) = f_1(r(h_I')) \otimes f_2(h)$.
Let $m \colon (0,1] \times (0,1] \to (0,1]$ be multiplication, and let $m^* \colon C_0((0,1]) \to C_0((0,1]) \otimes
C_0((0,1])$ be the induced homomorphism. Lemma~\ref{l:mcomm} yields positive
functions $f_1,f_2 \in C_0((0,1])$ such that $f_1(1) = 1$ and
\begin{equation}
f_1 \otimes f_2 \leq m^\ast(f) \in C_0((0,1]) \otimes C_0((0,1]).
\end{equation}
 As $\Phi_I(m^\ast(f)) = f(r(h_I') \otimes h)$, it follows that
\begin{equation}
f(h_I) = \kappa(\Phi_I(m^\ast(f))) \geq  \kappa( \Phi_I(f_1\otimes f_2)) = \kappa (f_1(r(h_I')) \otimes f_2(h)),
\end{equation}
for $I \in \mathcal{UP}(\mathcal I_0)$. As $(f_1 \circ r) \cdot g_{n_0} = g_{n_0}$ it follows that $K$ contains the element
\begin{equation}\label{eq:l:upwardsmap.neweq}
\sum_{\substack{I \in \mathcal{UP}(\mathcal I_0)\\ I \subseteq J_2}} \kappa(g_{n_0}(h_I') \otimes f_2(h)).
\end{equation}
As $h$ has spectrum $[0,1]$, the element $f_2(h) \in \mathcal O_\infty$ is non-zero and thus full.
Hence by the defining property of $\kappa$, the element~\eqref{eq:l:upwardsmap.neweq} generates
the same ideal as $\sum_{I \in \mathcal{UP}(\mathcal{I}_0),\,I \subseteq J_2} g_{n_0}(h'_I)$. Using
first that the $h'_I$ are pairwise orthogonal, and then the
definition~\eqref{eq:l:upwardsmap.neweq.2} of $h_{J_2}''$ we calculate:
\begin{equation}\label{eq:l:upwardsmap.neweq3}
\sum_{\substack{I \in \mathcal{UP}(\mathcal I_0)\\ I \subseteq J_2}} g_{n_0}(h_I') = g_{n_0}\Big( \sum_{\substack{I \in \mathcal{UP}(\mathcal I_0)\\ I \subseteq J_2}} h_I' \Big) = g_{n_0}(h_{J_2}'').
\end{equation}
In particular, $g_{n_0}(h''_{J_2}) \in K$. As $n_0 \geq n_{J_1,J_2}$, it follows from
\eqref{eq:J1ingdelta} that
\begin{equation}
\Theta(J_1) \subseteq \overline{B g_{n_0}(h_{J_2}'') B} \subseteq K.\qedhere
\end{equation}
\end{proof}


We now give our main technical existence result realising $\Cu$-morphisms by $^*$-homomorphisms
from cones. In the proof, we will use that a finite lattice has a unique minimum element (in our
case $0$), and again take the convention that the supremum of the empty set is this minimum
element. See Section \ref{S2} for further details.

\begin{lemma}\label{lm:Main}
Let $A$ and $B$ be $\Cstar$-algebras with $A$ separable and $B \otimes \mathcal O_\infty \cong B$,
and let $\Theta \colon \mathcal I(A) \to \mathcal I(B)$ be a $\Cu$-morphism. Then for $n\in\mathbb N$, there are finite
dimensional $\Cstar$-algebras $F_n$ and cpc maps $\psi_n \colon C_0((0,1]) \otimes A \to F_n$ and
$\eta_n \colon F_n \to B$  such that, writing
\begin{align}
\psi &\coloneqq (\psi_n)_{n=1}^\infty  \colon C_0((0,1]) \otimes A \to \frac{\prod_{n=1}^\infty  F_n}{\bigoplus_{n=1}^\infty  F_n},\quad\text{ and }\label{eq:psidef}\\
\eta &\coloneqq (\eta_n)_{n=1}^\infty  \colon \frac{\prod_{n=1}^\infty  F_n}{\bigoplus_{n=1}^\infty  F_n} \to B_{\seq}\label{eq:etadef}
\end{align}
for the induced maps, the following are satisfied:
\begin{itemize}
\item[$(a)$] $\psi$ is a $^*$-homomorphism;
\item[$(b)$] each $\eta_n$ is order zero; and
\item[$(c)$] the $^*$-homomorphism $\rho \colon C_0((0,1]) \otimes A \to B_{\seq}$ induced
    by the cpc order zero map $(\eta \circ \psi)(\id_{(0,1]} \otimes \cdot)$ (see
    Proposition~\ref{prop.Cone}) satisfies
\begin{equation}
\mathcal I(\rho)(I \otimes J) = \overline{B_{\seq} \Theta(J) B_{\seq}}
\end{equation}
for any $J\in \mathcal I(A)$ and any non-zero $I \in \mathcal I(C_0((0,1]))$.
\end{itemize}
If every quotient of $A$ is quasidiagonal, we can additionally arrange that
$\psi_n(C_0((0,1))\otimes A) = 0$ for each $n\in \mathbb N$.
\end{lemma}

\begin{proof}
Lemma~\ref{l:basis} yields a countable basis $(J_n)_{n=1}^\infty$ for $\mathcal I(A)$ such that
for each $J\in \mathcal I(A)$ there exist $k_1 < k_2< \dots$ satisfying $J_{k_1} \Subset J_{k_2}
\Subset\dots$ and $J = \overline{\bigcup_{n=1}^\infty J_{k_n}}$. For each $n$, let $\mathcal I_n$
be the sublattice of $\mathcal I(A)$ generated by $0, A, J_1, \dots,J_n$.  Since the ideal lattice
of a $\Cstar$-algebra is distributive (it is isomorphic to the lattice of open subsets of its
primitive-ideal space; see \cite[Theorem~4.1.3]{P:Book}), each $\mathcal I_n$ is finite. Clearly
$\mathcal I_n \subseteq \mathcal I_{n+1}$ for all $n$.

Let $\id_{(0,1]} \in \mathcal F_1 \subseteq \mathcal F_2 \subseteq \dots \subseteq C_0((0,1])$ be
a sequence of finite subsets with $\overline{\bigcup_i \mathcal F_i} = C_0((0,1])$. Let $\mathcal
G_1 \subseteq \mathcal G_2 \subseteq \dots \subseteq A$ be a sequence of finite sets of positive
contractions such that for each $k$,
\begin{equation}\label{eq:Jcapdense}
J_k \cap \Big( \bigcup_{n=1}^\infty \mathcal G_n \Big)
\end{equation}
is dense in the set of positive contractions in $J_k$.

Fix $n\in \mathbb N$. Write $\mathcal F_n \otimes \mathcal G_n \coloneqq  \{ f\otimes a : f\in
\mathcal F_n \text{ and } a\in \mathcal G_n\}$. Using the notation of Definition \ref{DefUPD}, for $I \in \mathcal{UP}(\mathcal I_n)$ let $P_n(I)$
denote the unique predecessor of $I$ in $\mathcal I_n$. For such an $I$, use
Lemma~\ref{l:conetomatrix} to pick an integer $N(n, I)$ and a cpc~map $\psi_{n,I} \colon
C_0((0,1]) \otimes A \to M_{N(n,I)}$ that is $(\mathcal F_n \otimes \mathcal
G_n,1/n)$-multiplicative and has the property that for all $f\in \mathcal F_n$ and $a \in \mathcal
G_n$,
\begin{align}
|f(1)| \, \| a+{}& M_A(I,P_n(I)) \|_{A/M_A(I,P_n(I))} - \tfrac{1}{n}\nonumber\\
    &\leq \| \psi_{n,I} (f \otimes a) \| \nonumber\\
    &\leq \| f\| \, \| a +  M_A(I,P_n(I)) \|_{A/M_A(I,P_n(I))}  + \tfrac{1}{n}.\label{eq:psinormp}
\end{align}
If every quotient of $A$ is quasidiagonal, Lemma~\ref{l:conetomatrix} allows us to choose each
$\psi_{n,I}$ so that additionally $\psi_{n,I}(C_0((0,1))\otimes A) = 0$.

Let  $F_n \coloneqq  \bigoplus_{I\in \mathcal{UP}(\mathcal I_n)} M_{N(n,I)}$ and let $\psi_n \colon
C_0((0,1])\otimes A \to F_n$ be the cpc map such that
\begin{equation}\label{eq:psindef}
\psi_n(x) = \bigoplus_{I\in \mathcal{UP}(\mathcal I_n)}   \psi_{n,I}(x),\quad x\in C_0((0,1])\otimes A.
\end{equation}
If every quotient of $A$ is quasidiagonal, then $\psi_n(C_0((0,1)) \otimes A) = 0$.

Let $(h_{n,I})_{I \in \mathcal{UP}(\mathcal I_n)}$ be a collection of pairwise orthogonal, positive
elements of norm 1 in $B$, satisfying conditions $(a)$ and $(b)$ of Lemma~\ref{l:upwardsmap}.
Choose embeddings $\iota_{n,I} \colon M_{N(n, I)} \to \mathcal O_\infty$, and for each $I$, let
$\eta_{n,I}' \colon M_{N(n,I)} \to B \otimes \mathcal O_\infty$ be the cpc~order zero map such
that
\begin{equation}\label{eq:etan'def}
\eta_{n,I}'(x) = h_{n,I} \otimes \iota_{n,I}(x), \quad x \in M_{N(n,I)}.
\end{equation}
For a positive contraction $f\in C_0((0,1])$, we have
\begin{equation}\label{eq:etan'def.new}
f(\eta_{n,I}')(x)=f(h_{n,I})\otimes \iota_{n,I}(x),\quad x\in M_{N(n,I)}.
\end{equation}

By Lemma~\ref{lem:kappa} there exists an isomorphism $\kappa \colon B\otimes \mathcal O_\infty \to
B$ such that $\mathcal I(\kappa)(K\otimes\mathcal O_\infty)=K$ for all $K\in\mathcal I(B)$. Define
$\eta_{n,I} \coloneqq  \kappa \circ \eta_{n,I}'\colon M_{N(n,I)}\rightarrow B$ and
\begin{equation}\label{eq:etandef}
\eta_n \coloneqq \bigoplus_{I\in \mathcal{UP}(\mathcal I_n)} \eta_{n,I} = \kappa \circ \Big( \bigoplus_{I\in \mathcal{UP}(\mathcal I_n)} \eta_{n,I}' \Big) \colon F_n \to B.
\end{equation}
Since, if every quotient of $A$ is quasidiagonal, we have chosen the $\psi_n$ such that
$\psi_n(C_0((0,1)) \otimes A) = 0$, it now suffices to show that the $\psi_n$ and $\eta_n$ satisfy
$(a)$--$(c)$. Since the $\eta_{n, I}$ are cpc~order zero maps with pairwise orthogonal
ranges, each $\eta_n$ is cpc~order zero, which is~$(b)$.

Let $\psi \colon C_0((0,1]) \otimes A \to \frac{\prod_{n=1}^\infty F_n}{\bigoplus_{n=1}^\infty
F_n}$ and $\eta \colon \frac{\prod_{n=1}^\infty F_n}{\bigoplus_{n=1}^\infty F_n} \to B_{\seq}$
be as in \eqref{eq:psidef}~and~\eqref{eq:etadef}. Then $\psi$ is a cpc~map, and $\eta$ is a
cpc~map of order zero. As each $\psi_n$ is $(\mathcal F_n\otimes\mathcal G_n,1/n)$-multiplicative, it
follows that $\bigcup_{n=1}^\infty  \mathcal F_n \otimes \mathcal G_n$ is contained in the
multiplicative domain of $\psi$, and hence $\psi$ is a $^*$-homomorphism. This proves~$(a)$, so it
remains to prove~$(c)$.

As $\eta \circ \psi \colon C_0((0,1]) \otimes A \to B_{\seq}$ is a cpc~map of order zero, we
let
\begin{equation}
\rho_{\eta \circ \psi} \colon C_0((0,1]) \otimes C_0((0,1]) \otimes A \to B_{\seq}
\end{equation}
be the induced $^*$-homomorphism given by Proposition~\ref{prop.Cone}. Then
\begin{equation}\label{eq:rhofact}
\rho_{\eta \circ \psi}(\id_{(0,1]} \otimes x) = \eta \circ \psi(x),\quad\text{ $x\in C_0((0,1])\otimes A$.}
\end{equation}
Let $m^\ast \colon C_0((0,1]) \to C_0((0,1]) \otimes C_0((0,1])$ be the $^*$-homomorphism induced
by the multiplication map $m \colon (0,1] \times (0,1] \to (0,1]$, and let $\rho$ denote the
composition
\begin{equation}\label{eq:rhocomp}
C_0((0,1]) \otimes A \xrightarrow{m^\ast \otimes \id_A} C_0((0,1]) \otimes C_0((0,1]) \otimes A \xrightarrow{\rho_{\eta \circ \psi}} B_{\seq}.
\end{equation}
Since $m^\ast(\id_{(0,1]}) = \id_{(0,1]} \otimes \id_{(0,1]}$, for each $a \in A$, we have
\begin{equation}\label{eq:newname1}
\rho(\id_{(0,1]} \otimes a) = \rho_{\eta \circ \psi}(\id_{(0,1]} \otimes \id_{(0,1]} \otimes a) = (\eta \circ \psi)( \id_{(0,1]} \otimes a).
\end{equation}
Hence $\rho$ is the $^*$-homomorphism induced by the cpc order zero map $(\eta \circ
\psi)(\id_{(0,1]} \otimes \cdot)$ as in Proposition~\ref{prop.Cone}.

Now fix a non-zero $I\in \mathcal I(C_0((0,1]))$, and $J \in \mathcal I(A)$. To prove~$(c)$,
it  suffices to prove the following two sublemmas.

\begin{sublemma}\label{sl:1}
With notation as in the proof of Lemma~\ref{lm:Main}, we have $\mathcal I(\rho)(I\otimes J)
\subseteq \overline{B_{\seq} \Theta(J) B_{\seq}}$.
\end{sublemma}

\begin{sublemma}\label{sl:2}
With notation as in the proof of Lemma~\ref{lm:Main}, we have $\overline{B_{\seq} \Theta(J)
B_{\seq}} \subseteq \mathcal I(\rho)(I\otimes J)$.
\end{sublemma}

\begin{proof}[Proof of Sublemma~\ref{sl:1}]
Since $\mathcal I(\rho)(I \otimes J) \subseteq \mathcal I(\rho)(C_0((0,1]) \otimes J)$, it is
enough to show that $\mathcal I(\rho)(C_0((0,1]) \otimes J) \subseteq \overline{B_{\seq}
\Theta(J) B_{\seq}}$.

By choice of $(J_n)_{n=1}^\infty$ there exist $k_1 < k_2 < \dots$ such that $J_{k_1} \Subset
J_{k_2} \Subset\dots$ and $J = \overline{\bigcup_{n=1}^\infty J_{k_n}}$. As $\mathcal I(\rho)$
preserves suprema, it suffices to show that each $\mathcal I(\rho)(C_0((0,1]) \otimes J_{k_n})
\subseteq \overline{B_{\seq} \Theta(J) B_{\seq}}$. Fix $n \in \mathbb N$ and let $K
\coloneqq J_{k_n}$.

As $K \Subset J$ and $\Theta$ is a $\Cu$-morphism, $\Theta(K) \Subset \Theta(J)$. As $B$ is
$\mathcal O_\infty$-stable it is strongly purely infinite, and hence weakly purely infinite by
\cite[Theorem~9.1]{KR-adv}. Thus \cite[Proposition~6.5]{Gabe17} gives
$\Theta(K)_{\seq}\subseteq \overline{B_{\seq} \Theta(J) B_{\seq}}$. So it suffices to
show that
\begin{equation}
\mathcal I(\rho)(C_0((0,1]) \otimes K) \subseteq \Theta(K)_{\seq}.
\end{equation}
Since $K = J_{k_n}$, our choice of the $(\mathcal G_m)_{m=1}^\infty$ ensures that $K \cap \left(
\bigcup_{m=1}^\infty \mathcal G_m \right)$ is dense in the set of positive contractions in $K$.
Since $\id_{(0,1]} \in C_0((0,1])$ is a full element, it therefore suffices to show that for each
$m \ge k_n$ and $a \in K \cap \mathcal G_m$,
\begin{equation}\label{eq:rhoinKinfty}
\rho(\id_{(0,1]} \otimes a) \in \Theta(K)_{\seq}.
\end{equation}
So fix such $m$ and $a$.

By \eqref{eq:newname1}, the element $\rho(\id_{(0,1]} \otimes a)$ is represented by the sequence
$(\eta_r \circ \psi_r(\id_{(0,1]} \otimes a))_{r=1}^\infty \in \prod_{r=1}^\infty B$. For $r\geq
m$, let
\begin{align} \label{eq:xkdef}
x_r &\coloneqq  \sum_{\substack{L\in \mathcal{UP}(\mathcal I_r) \\ L \subseteq K}} \eta_{r,L} \circ \psi_{r,L} (\id_{(0,1]} \otimes a),\text{ and}\nonumber\\
y_r &\coloneqq  \sum_{\substack{L\in \mathcal{UP}(\mathcal I_r) \\ L \not \subseteq K}} \eta_{r,L} \circ \psi_{r,L} (\id_{(0,1]} \otimes a).
\end{align}
Then
\begin{equation}
\eta_r \circ \psi_r(\id_{(0,1]} \otimes a) = x_r + y_r.
\end{equation}

Since $(\eta_{r,L})_{L\in \mathcal{UP}(\mathcal I_r)}$ is a family of contractive maps with pairwise
orthogonal images, $x_r$ and $y_r$ are contractions. For $L \in \mathcal{UP}(\mathcal I_r)$ such
that $L \not \subseteq K$, contractivity of $\eta_{r,L}$ and~\eqref{eq:psinormp} show that
$\|\eta_{r,L} \circ \psi_{r,L} (\id_{(0,1]} \otimes a)\| \le \frac{1}{r} + \| a + M_A(L,P_r(L))
\|_{A/M_A(L,P_r(L))}$, and so
\begin{align}
\| y_r\|
    &= \max \{ \| \eta_{r,L} \circ \psi_{r,L} (\id_{(0,1]} \otimes a)\| : L \in \mathcal{UP}(\mathcal I_r), \, L \not \subseteq K\} \nonumber\\
    &\leq \frac{1}{r} + \max\{  \| a +M_A(L,P_r(L)) \|_{A/M_A(L,P_r(L))} : L \in \mathcal{UP}(\mathcal I_r), \, L \not \subseteq K\}.
\end{align}
As $K = J_{k_n}$ and $r \geq m \geq k_n$, we have $K \in \mathcal I_r$. Lemma~\ref{l:JinKPK} gives
$K \subseteq M_A(L,P_r(L))$ for any $L \in \mathcal{UP}(\mathcal I_r)$ satisfying $L \not \subseteq K$.
Hence $\| a +M_A(L,P_r(L))\|_{A/M_A(L,P_r(L))} =0$ for any such $L$. Thus $\| y_r \| \leq 1/r$.
Consequently $\rho(\id_{(0,1]} \otimes a)$ is represented by the sequence $(x_r)_{r=1}^\infty$.

Suppose that $r \geq m$ and that $L \in \mathcal{UP}(\mathcal I_r)$ satisfies $L \subseteq K$.
Recall from \eqref{eq:etan'def} that $\eta_{r,L}'(z) = h_{r,L} \otimes \iota_{r,L}(z)$
for every $z\in M_{N(r,L)}$, where $h_{r,L} \in \Theta(L) \subseteq \Theta(K)$. Thus
\begin{equation}\label{eq:etakL'psi}
\eta_{r,L}' \circ \psi_{r,L}  (\id_{(0,1]} \otimes a) \in \Theta(K) \otimes \mathcal O_\infty.
\end{equation}
Recall that $\kappa \colon B \otimes \mathcal O_\infty \xrightarrow \cong B$ satisfies $\mathcal
I(\kappa)(K'\otimes\mathcal O_\infty)=K'$ for $K'\in\mathcal I(B)$. Since $\eta_{r,L} = \kappa
\circ \eta_{r,L}'$ it follows from \eqref{eq:etakL'psi} that
\begin{equation}
\eta_{r,L} \circ \psi_{r,L}  (\id_{(0,1]} \otimes a) \in \Theta(K).
\end{equation}
Hence for $r\geq m$, we have $x_r \in \Theta(K)$ (see \eqref{eq:xkdef}). Since $\rho(\id_{(0,1]}
\otimes a)$ is represented by the sequence $(x_r)_{r=1}^\infty$ it follows that
$\rho(\id_{(0,1]}\otimes a)$ belongs to
$\Theta(K)_{\seq}$, as required.\let\oldqed\qed\renewcommand{\qed}{\oldqed\quad (Sublemma~\ref{sl:1}).}
\end{proof}

\begin{proof}[Proof of Sublemma~\ref{sl:2}]
Recall that we have fixed $I\in\mathcal I(C_0((0,1]))$ non-zero, and $J\in\mathcal I(A)$. Let $\iota
\colon B \to B_{\seq}$ be the canonical inclusion. Then
\begin{equation}
\mathcal I(\iota) \circ \Theta(K) = \overline{B_{\seq} \Theta(K) B_{\seq}},\quad K \in \mathcal I(A).
\end{equation}
As $A$ is separable, $J = \sup\{J_0 \in \mathcal I(A) : J_0 \Subset J\}$ by
\cite[Corollary~2.3]{Gabe17}. Fix $J_0 \in \mathcal I(A)$ such that $J_0 \Subset J$. Since
$\mathcal I(\iota) \circ \Theta$ preserves suprema, it suffices to show that
\begin{equation}
\overline{B_{\seq} \Theta(J_0) B_{\seq}}  \subseteq \mathcal I(\rho)(I\otimes J).
\end{equation}

By \eqref{eq:rhocomp}, it follows from \cite[Proposition~2.15]{Gabe17} that
\begin{equation}
\mathcal I(\rho) = \mathcal I(\rho_{\eta \circ \psi}) \circ \mathcal I(m^\ast \otimes \id_A).
\end{equation}
Hence
\begin{equation}\label{eq:Irhocomp}
\mathcal I(\rho) (I \otimes J) = \mathcal I(\rho_{\eta \circ \psi})( \mathcal I(m^\ast)(I) \otimes J).
\end{equation}
As $I\neq 0$ by assumption, Lemma~\ref{l:mcomm} yields positive, non-zero contractions $f,g\in
C_0((0,1])$ such that $f\otimes g \in \mathcal I(m^\ast)(I)$, and $g(1) = 1$. We may assume
without loss of generality that $\| f\| = 1$. Let $a\in J$ be a full, positive contraction.
Equation~\eqref{eq:Irhocomp} gives
\begin{equation}
\rho_{\eta\circ \psi}(f\otimes g \otimes a) \in\mathcal I(\rho_{\eta \circ \psi})( \mathcal I(m^\ast)(I) \otimes J) = \mathcal I(\rho) (I \otimes J),
\end{equation}
so it suffices to show that
\begin{equation}\label{eq:binrhoetapsi}
\overline{B_{\seq}\Theta(J_0)B_{\seq}} \subseteq \overline{B_{\seq} \rho_{\eta\circ \psi}(f\otimes g \otimes a) B_{\seq}}.
\end{equation}
Recall that $\psi$ and $\eta$ are the maps defined in \eqref{eq:psidef}~and~\eqref{eq:etadef}, and
that $\rho_{\eta \circ \psi}$ is the $^*$-homomorphism induced by the composition $\eta \circ
\psi$. Let $\rho_\eta$ be the $^*$-homomorphism induced by $\eta$ as in
Proposition~\ref{prop.Cone}. Since $\psi$ is a $^*$-homomor\-phism and $\eta$ is a cpc~order zero
map, Lemma~\ref{l:compoz} implies that $\rho_{\eta \circ \psi}$ is equal to the composition
\begin{equation}
C_0((0,1]) \otimes C_0((0,1]) \otimes A \xrightarrow{\id_{C_0((0,1])} \otimes \psi} C_0((0,1]) \otimes \frac{\prod_{n=1}^\infty  F_n}{\bigoplus_{n=1}^\infty  F_n} \xrightarrow{\rho_\eta} B_{\seq}.
\end{equation}
Hence, with $f(\eta)$ as in~\eqref{fncaloz},
\begin{equation}
\rho_{\eta\circ \psi}(f\otimes g \otimes a) = \rho_\eta( f \otimes \psi(g\otimes a)) =  f(\eta)(\psi(g \otimes a)).
\end{equation}
As $\eta$ is induced by the sequence $(\eta_n)_{n=1}^\infty$ of cpc~order zero maps,
Lemma~\ref{l:fctcalcseqoz} implies that $f(\eta)$ is induced by the sequence
$(f(\eta_n))_{n=1}^\infty$ of cpc~order zero maps. As $\psi$ is represented by the sequence
$(\psi_n)_{n=1}^\infty$, it follows that $\rho_{\eta\circ \psi}(f\otimes g \otimes a)$ is
represented by the sequence
\begin{equation}\label{eq:fetapsi}
(f(\eta_n)(\psi_n(g\otimes a)))_{n=1}^\infty.
\end{equation}

Fix $n\in \mathbb N$ and $x = (x_K)_{K\in \mathcal{UP}(\mathcal I_n)} \in F_n = \bigoplus_{K \in
\mathcal{UP}(\mathcal I_n)} M_{N(n,K)}$. By~\eqref{eq:etandef}, we have
\begin{equation}
f(\eta_n)(x) = f\Big(\kappa \circ \bigoplus_{K \in \mathcal{UP}(\mathcal I_n)} \eta_{n,K}' \Big) (x).
\end{equation}
As $\kappa$ is an isomorphism, we obtain $f(\eta_n)(x) = \kappa \big( f\big(\bigoplus_{K \in
\mathcal{UP}(\mathcal I_n)} \eta_{n,K}' \big) (x_K) \big)$. For fixed $n$, the maps $(\eta_{n,K})_{K
\in \mathcal{UP}(\mathcal I_n)}$ have mutually orthogonal ranges, so Lemma~\ref{Lem:cpcRanges}
implies that $f(\eta_n)(x) = \sum_{K \in \mathcal{UP}(\mathcal I_n)} \kappa (f(\eta_{n,K}')(x_K))$.
Applying~\eqref{eq:etan'def.new} now yields
\begin{equation}
f(\eta_n)(x) = \sum_{K \in \mathcal{UP}(\mathcal I_n)} \kappa (f(h_{n,K}) \otimes \iota_{n,K}(x_K)). \label{eq:bigsum}
\end{equation}
Combining~\eqref{eq:bigsum}, \eqref{eq:fetapsi}, and~\eqref{eq:psindef}, we deduce that
$\rho_{\eta\circ \psi}(f\otimes g \otimes a)$ is represented by the sequence
\begin{equation}\label{eq:rhoetapsirepresent}
\Big(\sum_{K \in \mathcal{UP}(\mathcal I_n)} \kappa (f(h_{n,K}) \otimes \iota_{n,K}(\psi_{n,K}(g \otimes a))) \Big)_{n=1}^\infty \in \prod_{n=1}^\infty B.
\end{equation}

By choice of $(J_k)_{k=1}^\infty$, there exist $k_1 < k_2 < \dots$ such that $J_{k_{l-1}} \Subset
J_{k_l}$ for all $l$, and $J = \overline{\bigcup J_{k_l} }$. Since $J_0 \Subset J$, there exists
$l\in \mathbb N$ such that
\begin{equation}\label{eq:JcpctJ}
J_0 \subseteq J_{k_{l-1}}.
\end{equation}
As $a\in J$ is a full, positive contraction, Lemma~\ref{l:alowerbound} applied to $J_{k_l}\Subset
J$ provides $\lambda >0$, such that for any ideals $K_1 \subsetneq K_2 \subseteq J_{k_l}$, we have
\begin{equation}\label{eq:a+K}
\| a + M_A(K_2,K_1)\|_{A/ M_A(K_2,K_1)} \geq \lambda.
\end{equation}
Pick $k_0 \geq \max\{ k_l, 10/\lambda\}$ such that there exist $g_0 \in \mathcal F_{k_0}$ and $a_0
\in \mathcal G_{k_0}$ satisfying
\begin{equation}\label{eq:normg0a0}
\| g_0 - g\| \leq \lambda/10 \quad\text{ and }\quad \| a_0 - a\| \leq \lambda/10.
\end{equation}
We chose $g$ satisfying $g(1) = 1$, so $|g_0(1) - 1| \leq \lambda/10$. For $n\geq k_0$ and $K\in
\mathcal{UP}(\mathcal I_n)$, we calculate, using the fact that $a_0$ is a contraction and $\psi_{n,K}$ is contractive at the first step,
\begin{align}
\| \psi_{n,K}&(g \otimes a) \|\nonumber\\
    &\geq \| \psi_{n,K}(g_0 \otimes a_0)\| - \frac{2\lambda}{10} \quad&&\text{by~\eqref{eq:normg0a0}}\nonumber\\
    &\geq |g_0(1)| \, \big\| a_0 + M_A(K,P_n(K))\big\|_{A/M_A(K,P_n(K))} - \frac{3\lambda}{10}\quad&&\text{by~\eqref{eq:psinormp}} \nonumber\\
    &\geq \| a + M_A(K,P_n(K))\|_{A/M_A(K,P_n(K))} - \frac{\lambda}{2}\quad&&\text{by~\eqref{eq:normg0a0}}. \label{eq:psia+K}
\end{align}
As $k_0 \geq k_l$, it follows that $J_{k_{l-1}} , J_{k_l} \in \mathcal I_n$ for every $n\geq k_0$.
Hence, whenever $n \geq k_0$ and $K \in \mathcal{UP}(\mathcal I_n)$ satisfies $K \subseteq J_{k_l}$,
applying~\eqref{eq:psia+K} and then~\eqref{eq:a+K}, we obtain
\begin{align}
\|\psi_{n,K}(g \otimes a) \|
    &\geq \| a + M_A(K,P_n(K))\|_{A/M_A(K,P_n(K))} - \lambda/2 \nonumber\\
    &\geq \lambda - \lambda/2 \nonumber\\
    &\geq \lambda/2. \label{eq:lambda/2}
\end{align}

Recall that $\psi_{n,K}$ takes values in the matrix algebra $M_{N(n,K)}$. Thus,
by~\eqref{eq:lambda/2}, for any $n \geq k_0$ and $K \in \mathcal{UP}(\mathcal I_n)$ satisfying $K
\subseteq J_{k_l}$ there exists a non-zero projection $p_{n,K} \in M_{N(n,K)}$, such that
\begin{equation}\label{eq:psilambda/2}
\psi_{n,K} (g\otimes a) \geq p_{n,K} \cdot \lambda/2.
\end{equation}

Recall from~\eqref{eq:rhoetapsirepresent} that  $\rho_{\eta \circ \psi}(f\otimes g \otimes a)$ is
represented by the sequence $\big(\sum_{K\in \mathcal{UP}(\mathcal I_n)} \kappa( f(h_{n,K}) \otimes
\iota_{n,K}(\psi_{n,K}(g\otimes a)))\big)^\infty_{n=1}$.  For each $n\geq k_0$,
Equation~\eqref{eq:psilambda/2} implies that
\begin{align}
\sum_{\substack{K\in \mathcal{UP}(\mathcal I_n) \\ K \subseteq J_{k_l}}} \kappa&(f(h_{n,K}) \otimes \iota_{n,K}(p_{n,K}))\nonumber\\
&\leq  \frac{2}{\lambda}\sum_{K\in \mathcal{UP}(\mathcal I_n)} \kappa( f(h_{n,K}) \otimes \iota_{n,K}(\psi_{n,K}(g\otimes a))). \label{eq:sumkappafhnK}
\end{align}
So the element in $B_{\seq}$ represented by\footnote{We only define the sequence for $n\geq
k_0$, as such sequences still determine a unique element in $B_{\seq}$.}
\begin{equation}\label{eq:sumXifp}
\Big( \sum_{\substack{K\in \mathcal{UP}(\mathcal I_n) \\ K \subseteq J_{k_l}}} \kappa (f(h_{n,K}) \otimes \iota_{n,K}(p_{n,K})) \Big)_{n=k_0}^\infty
\end{equation}
belongs to the ideal generated by $\rho_{\eta \circ \psi}(f \otimes g \otimes a)$. Thus, to
establish~\eqref{eq:binrhoetapsi} and complete the proof of the sublemma, it suffices to show that
$\Theta(J_0)$ is contained in the ideal of $B_{\seq}$ generated by the element represented
by~\eqref{eq:sumXifp}.

Recall that $f\in C_0((0,1])$ is a positive contraction with $\|f\|= 1$. Fix $f_0 \colon [0,1] \to
[0,1]$ continuous with $f_0(0)=0$ and $f_0(1)=1$. Then $f_0 \circ f \colon [0,1] \to [0,1]$ is a
non-zero, continuous function such that $f_0 \circ f(0)=0$. Fix $n\geq k_0$. Using first that the
$f(h_{n, K}) \otimes \iota_{n,K}(p_{n,K})$ are mutually orthogonal, and then that $\iota_{n,
K}(p_{n,K})$ is a projection, we calculate:
\begin{align}
f_0\Big( \sum_{\substack{K\in \mathcal{UP}(\mathcal I_n) \\ K \subseteq J_{k_l}}} \kappa&(f(h_{n,K}) \otimes \iota_{n,K}(p_{n,K})) \Big) \label{eq:f0sumstuff}\\
&= \sum_{\substack{K\in \mathcal{UP}(\mathcal I_n) \\ K \subseteq J_{k_l}}} \kappa (f_0 ( f(h_{n,K}) \otimes \iota_{n,K}(p_{n,K}))) \nonumber\\
&= \sum_{\substack{K\in \mathcal{UP}(\mathcal I_n) \\ K \subseteq J_{k_l}}} \kappa ((f_0 \circ f)(h_{n,K}) \otimes \iota_{n,K}(p_{n,K})).\label{eq:f0fstuff}
\end{align}
As the $\iota_{n,K}(p_{n,K}) \in \mathcal O_\infty$ are non-zero projections, and as $\mathcal
O_\infty$ is simple, the defining property of $\kappa$ ensures that the element
in~\eqref{eq:f0fstuff} generates the same ideal as
\begin{equation}\label{eq:sumf0fhnK}
\sum_{\substack{K\in \mathcal{UP}(\mathcal I_n) \\ K \subseteq J_{k_l}}} (f_0 \circ f)(h_{n,K}).
\end{equation}
Since $f_0 \circ f \colon [0,1] \to [0,1]$ is non-zero, continuous and maps 0 to 0, and as
$J_{k_{l-1}} \Subset J_{k_l}$ by \eqref{eq:JcpctJ}, Lemma~\ref{l:upwardsmap}(b) implies that
$\Theta(J_{k_{l-1}})$ is contained in the ideal generated by the element of~\eqref{eq:sumf0fhnK} by
choice of $h_{n,K}$ (see the text just above~\eqref{eq:etan'def}). Hence $\Theta(J_{k_{l-1}})$ is
contained in the ideal generated by the element of~\eqref{eq:f0sumstuff}.

Now fix a positive $b \in \Theta(J_0)$ with $\| b \|=1$. To show that $\Theta(J_0)$ is contained in the
ideal generated by the element $c$ of $B_{\seq}$ represented by~\eqref{eq:sumXifp}, it
suffices to show that $b$ belongs to $\overline{B_{\seq} c B_{\seq}}$. As $B$ is $\mathcal
O_\infty$-stable, it is purely infinite\footnote{Recall that a $\Cstar$-algebra $D$ is
\emph{purely infinite} if it has no characters, and if whenever $d_1,d_2 \in D$ are positive such
that $d_1 \in \overline{Dd_2D}$ then $d_1 \precsim d_2$.} by \cite[Proposition~4.5]{KR-AJM}. Since
$b$ is a positive contraction in $\Theta (J_0)\Subset \Theta(J_{k_{l-1}})$, it is Cuntz dominated by the element in
\eqref{eq:f0sumstuff}.  The element
\begin{equation}
\sum_{\substack{K\in \mathcal{UP}(\mathcal I_n) \\ K \subseteq J_{k_l}}} \kappa (f(h_{n,K}) \otimes \iota_{n,K}(p_{n,K}))
\end{equation}
is a positive contraction of norm $1$ by Lemma~\ref{l:upwardsmap}, since all $h_{n,K}$ are
pairwise orthogonal, positive elements with spectrum $[0,1]$, and since $f\in C_0((0,1])$ is a
positive contraction with $\|f\| = 1$ (see definition of $f$ right after \eqref{eq:Irhocomp}).
Since $f_0\colon [0,1] \to [0,1]$ is continuous and satisfies $f_0(0) = 0$ and $f_0(1) = 1$,
Lemma~\ref{l:Cuntzdompi} gives a contraction $z_n \in B$ such that
\begin{equation}\label{eq:zconj}
\Big\| b -  z_n^\ast  \Big( \sum_{\substack{K\in \mathcal{UP}(\mathcal I_n) \\ K \subseteq J_{k_l}}} \kappa (f(h_{n,K}) \otimes \iota_{n,K}(p_{n,K})) \Big) z_n \Big\| < \frac{1}{n}.
\end{equation}

Let $z \in B_{\seq}$ be the element induced by $(z_n)_{n=k_0}^\infty$. By~\eqref{eq:zconj}, we
have $\iota(b) = z^\ast c z$ as required.\let\oldqed\qed\renewcommand{\qed}{\oldqed\quad
(Sublemma~\ref{sl:2}).}
\end{proof}

With the two sublemmas in place, the proof of Lemma~\ref{lm:Main} is complete.
\end{proof}


\section{$\mathcal O_2$- and $\mathcal O_\infty$-stable $^*$-homomorphisms}\label{S4}

Our main objective in this section is to upgrade the technical construction of the previous
section to additionally ensure that $\rho$ is $\mathcal O_2$-stable. This is
Lemma~\ref{l:O2stablelift} below.

We begin with a review of $\mathcal O_2$ and $\mathcal O_\infty$-stability for maps. We use a
sequential version of a relative commutant construction developed by Kirchberg in the setting of
ultrapowers \cite{KirAbel}. If $\theta \colon A \to B$ is a $^*$-homomorphism, we denote the
relative commutant of $\theta(A) \subseteq B \subseteq B_{\seq}$  by $B_{\seq} \cap
\theta(A)'$. The annihilator,
\begin{equation}
\Ann \theta(A) \coloneqq \{ x\in B_{\seq} : x \theta(A) = \theta(A) x = \{0\} \},
\end{equation}
of $\theta(A)$ in $B_{\seq}$ is an ideal in $B_{\seq} \cap \theta(A)'$. If $A$ is
separable, then any sequential approximate identity $(e_n)^\infty_{n=1}$ for $A$ represents a unit
for $(B_{\seq} \cap \theta(A)') / \Ann \theta(A)$.  With this notation, we recall the
definitions of $\mathcal O_2$ and $\mathcal O_\infty$-stability of maps from
\cite[Definition~3.16]{Gabe17}.

\begin{definition}
Let $A$ and $B$ be $\Cstar$-algebras with $A$ separable, and let $\theta \colon A \to B$ be a
$^*$-homomorphism. Then $\theta$ is \emph{$\mathcal O_2$-stable} (respectively~\emph{$\mathcal
O_\infty$-stable}) if $\mathcal O_2$ (respectively~$\mathcal O_\infty$) embeds unitally into
\begin{equation}
\frac{B_{\seq} \cap \theta(A)'}{\Ann \theta(A)}.
\end{equation}
\end{definition}

In this framework, it goes back to work of Kirchberg, Lin, R\o{}rdam and Phillips (abstracted to
strongly self-absorbing algebras in \cite{TW:Trans}), that a separable $\Cstar$-algebra $A$ is
$\mathcal O_2$-stable (resp.~$\mathcal O_\infty$-stable) if and only if the identity map $\id_A$
is $\mathcal O_2$-stable (resp.~$\mathcal O_\infty$-stable), (see \cite[Proposition~3.19]{Gabe17}
for this exact statement).

The next lemma is extracted from the proof of \cite[Theorem~4.3]{BEMSW} to show that embeddings of
cones into simple purely infinite algebras are $\mathcal O_2$-stable.
\begin{lemma}\label{Lm:HO2}
Let $A$ be a simple, purely infinite $\Cstar$-algebra, and let $h\in A$ be a positive element with
spectrum $[0,1]$. Then the $^*$-homomorphism $\phi \colon C_0((0,1]) \to A$ given by functional
calculus on $h$ is $\mathcal O_2$-stable.
\end{lemma}
\begin{proof}
Fix $n\in \mathbb N$.  As $A$ has real rank zero (\cite{Z:PAMS}), find $0<\lambda_1\leq \dots \leq \lambda_k \leq 1$ and non-zero pairwise orthogonal projections $\tilde{p}_1,\dots,\tilde{p}_k\in A$, such that
\begin{equation}
\| h - \sum_{i=1}^k \lambda_i \tilde{p}_i\| < 1/(2n).
\end{equation}
Set $\lambda_0=0$, and notice that $\lambda_i-\lambda_{i-1}< 1/(2n)$ for $i=1,\dots,k$.

As $A$ is simple and purely infinite, find a non-zero subprojection $p_k$ of $\tilde{p}_k$ such that $[p_k]_0=0$ in $K_0(A)$.  Now find a non-zero subprojection $q_{k-1}$ of $\tilde{p}_{k-1}$ so that $[q_{k-1}]_0=-[\tilde{p}_k-p_k]_0$ in $K_0(A)$, and set $p_{k-1}=q_{k-1}+(\tilde{p}_k-p_k)$.  Next find a non-zero subprojection $q_{k-2}$ of $\tilde{p}_{k-2}$ so that $[q_{k-2}]_0=-[\tilde{p}_{k-1}-q_{k-1}]_0$ in $K_0(A)$ and set $p_{k-2}=q_{k-2}+(\tilde{p}_{k-1}-q_{k-1})$.  Carry on in this way to find $p_{k-3},\dots,p_1$ so that each $[p_i]_0=0$.  Then
\begin{equation}
\|\sum_{i=1}^k\lambda_ip_i - \sum_{i=1}^k\lambda_i\tilde{p}_i\|\leq \max_{i=1,\dots,k}(\lambda_i-\lambda_{i-1})<\frac{1}{2n},
\end{equation}
so that
\begin{equation}
\| h - \sum_{i=1}^k \lambda_i p_i\| < 1/n.
\end{equation}

As each $p_i$ is properly infinite with $[p_i]_0 = 0$, fix unital embeddings $\theta_{n,i} \colon \mathcal O_2 \to p_i A p_i$.  Let $\theta_n = \bigoplus_{i=1}^k \theta_{n,i} \colon \mathcal O_2 \to \bigoplus_{i=1}^k p_i Ap_i \subseteq A$. Then $\| [\theta_n(x),h]\| \to 0$ for any $x\in \mathcal O_2$, and $\theta_n(1_{\mathcal O_2}) h \to h$. Thus $(\theta_n)_{n=1}^\infty$ induces a unital embedding of $\mathcal O_2$ into $(B_{\seq} \cap \phi(C_0((0,1]))')/\Ann  \phi(C_0((0,1]))$.\end{proof}

As expected, the tensor product of an $\mathcal{O}_2$-stable (respectively
$\mathcal{O}_\infty$-stable) $^*$-homomorphism with another $^*$-homomorphism is again
$\mathcal{O}_2$-stable (respectively $\mathcal{O}_\infty$-stable).

\begin{lemma}\label{Lm:TensorO2}
Let $\phi \colon A\to B$ and $\psi \colon C \to D$ be $^*$-homomorphisms with $A$ and $C$
separable, and suppose that $\psi$ is $\mathcal O_2$-stable (or $\mathcal O_\infty$-stable). Then
$\phi \otimes \psi \colon A \otimes_{\max{}} C \to B \otimes_{\max{}} D$ is $\mathcal O_2$-stable
(or $\mathcal O_\infty$-stable).
\end{lemma}
\begin{proof}
The canonical $^*$-homomorphism $B_{\seq} \otimes_{\max{}} D_{\seq} \to (B\otimes_{\max{}}
D)_{\seq}$ induces a unital $^*$-homomorphism from the maximal tensor product
\begin{equation}
\Big(\frac{B_{\seq} \cap \phi(A)'}{\Ann \phi(A)}\Big) \otimes_{\max{}}
\Big(\frac{D_{\seq} \cap \psi(C)'}{\Ann \psi(C)}\Big)
\end{equation}
to the quotient
\begin{equation}
\frac{(B\otimes_{\max{}} D)_{\seq} \cap (\phi \otimes \psi)(A \otimes_{\max{}}
C)'}{\Ann (\phi \otimes \psi)(A\otimes_{\max{}} C)}.\label{Lm:TensorO2.1}
\end{equation}
Since $(B_{\seq}\cap\phi(A)')/\Ann(\phi(A))$ is unital, a unital embedding $\mathcal
O_2\rightarrow (D_{\seq}\cap \psi(C)')/\Ann(\psi(C))$ given by $\mathcal O_2$-stability of
$\psi$ gives rise to a unital embedding of $\mathcal O_2$ into the algebra of~(\ref{Lm:TensorO2.1}). So
$\phi\otimes\psi$ is $\mathcal O_2$-stable.  The same argument works with $\mathcal O_\infty$ in
place of $\mathcal O_2$.
\end{proof}

Following \cite{KR-AJM}, a non-zero, positive element $a\in A$ is called \emph{properly infinite}
if $a\oplus a \precsim a \oplus 0$ in $M_2(A)$. We record the following lemma for later use.

\begin{lemma}\label{l:imageOinftystable}
Any non-zero, positive element in the image of an $\mathcal O_\infty$-stable $^*$-homomorphism, is
properly infinite.
\end{lemma}
\begin{proof}
Let $\theta \colon A \to B$ be an $\mathcal O_\infty$-stable $^*$-homomorphism, and let $a\in A$
be a positive element such that $\theta(a)$ is non-zero. Let $s_1,s_2\in B_{\seq} \cap
\theta(A)'$ be elements so that $s_1+ \Ann \theta(A) , s_2 + \Ann \theta(A)$ are isometries in
$(B_{\seq} \cap \theta(A)')/ \Ann \theta(A)$ with mutually orthogonal range projections. Let
$(s_i^{(n)})_{n=1}^\infty \in \prod_{n=1}^\infty B$ be a lift of $s_i$ for $i=1,2$. Let $x_n
\coloneqq  \theta(a)^{1/4} s_1^{(n)} \theta(a)^{1/4}$ and $y_n \coloneqq  \theta(a)^{1/4} s_2^{(n)}
\theta(a)^{1/4}$. Then $x_n, y_n \in \overline{\theta(a) B \theta(a)}$, $x_n^\ast x_n \to \theta(a)$,
$y_n^\ast y_n \to \theta(a)$ and $x_n^\ast y_n \to 0$. By \cite[Proposition~3.3(iv)]{KR-AJM},
$\theta(a)$ is properly infinite.
\end{proof}

We are now able to upgrade our main technical lemma (Lemma~\ref{lm:Main}) to additionally insist
that the $^*$-homomorphism constructed is $\mathcal O_2$-stable.

\begin{lemma}\label{l:O2stablelift}
Let $A$ and $B$ be $\Cstar$-algebras with $A$ separable and $B \otimes \mathcal O_\infty \cong B$,
and let $\Theta \colon \mathcal I(A) \to \mathcal I(B)$ be a $\Cu$-morphism. Then for $n\in\mathbb N$, there are finite
dimensional $\Cstar$-algebras $F_n$ and cpc maps $\psi_n \colon C_0((0,1]) \otimes A \to F_n$ and
$\eta_n \colon F_n \to B$ such that, writing
\begin{align}
\psi &\coloneqq (\psi_n)_{n=1}^\infty  \colon C_0((0,1]) \otimes A \to \frac{\prod_{n=1}^\infty  F_n}{\bigoplus_{n=1}^\infty  F_n}, \quad\text{ and} \\
\eta &\coloneqq (\eta_n)_{n=1}^\infty  \colon \frac{\prod_{n=1}^\infty  F_n}{\bigoplus_{n=1}^\infty  F_n} \to B_{\seq}
\end{align}
for the induced maps, the following are satisfied:
\begin{itemize}
\item[$(a)$] $\psi$ is a $^*$-homomorphism;
\item[$(b)$] each $\eta_n$ is order zero;
\item[$(c)$] if $\rho \colon C_0((0,1]) \otimes A \to B_{\seq}$ is the $^*$-homomorphism
    induced by the cpc order zero map $(\eta \circ \psi)(\id_{(0,1]} \otimes \cdot)$ (see
    Proposition~\ref{prop.Cone}), then
\begin{equation}
\mathcal I(\rho)(I \otimes J) = \overline{B_{\seq} \Theta(J) B_{\seq}}
\end{equation}
for any $J\in \mathcal I(A)$ and any non-zero $I \in \mathcal I(C_0((0,1]))$; and
\item[$(d)$] the $^*$-homomorphism $\rho$ is $\mathcal O_2$-stable.
\end{itemize}
If every quotient of $A$ is quasidiagonal, we may additionally arrange that
$\psi_n(C_0((0,1))\otimes A) = 0$ for each $n\in \mathbb N$.
\end{lemma}
\begin{proof}
Apply Lemma~\ref{lm:Main} to obtain $(F_n, \psi_n, \eta_n')_{n=1}^\infty$ satisfying $(a)$--$(c)$
and satisfying $\psi_n(C_0((0,1)) \otimes A) = 0$ for all $n$ if every quotient of $A$ is
quasidiagonal. Fix a positive contraction $h\in \mathcal O_\infty$ with spectrum $[0,1]$. By
Lemma~\ref{lem:kappa}, there exists an isomorphism $\kappa \colon \mathcal O_\infty \otimes B \to
B$ such that $\mathcal I(\kappa)(\mathcal O_\infty\otimes K)=K$ for each $K\in\mathcal I(B)$.
Define $\eta_n \coloneqq  \kappa(h \otimes \eta_n'(\cdot)) \colon F_n \to B$. We will show that
$(F_n, \psi_n, \eta_n)_{n=1}^\infty$ satisfies conditions $(a)$--$(d)$ above.

Statement~$(a)$ and the final assertion are satisfied by choice of the $\psi_n$. Each $\eta'_n$ is
cpc~order zero, so each $\eta_n$ is cpc~order zero as well, giving~$(b)$. It remains to check
$(c)$~and~$(d)$. We first verify~$(d)$.

The map $\kappa$ induces a $^*$-homomorphism
\begin{equation}
\kappa_0 \colon \mathcal O_\infty \otimes B_{\seq} \to B_{\seq},
\end{equation}
such that, for $x\in \mathcal O_\infty$ and $(b_n)_{n=1}^\infty  \in B_{\seq}$,
\begin{equation}
\kappa_0 (x \otimes (b_n)_{n=1}^\infty) = (\kappa(x \otimes b_n))_{n=1}^\infty.
\end{equation}
Let
\begin{equation}
\eta' \coloneqq (\eta_n')_{n=1}^\infty \colon \frac{\prod_{n=1}^\infty  F_n}{\bigoplus_{n=1}^\infty  F_n} \to B_{\seq}.
\end{equation}
Then $\kappa_0 \circ (h \otimes \eta'(\cdot)) = \eta$.

Let $\rho'\colon C_0((0,1]) \otimes A \to B_{\seq}$ be the unique $^*$-homomorphism such that
\begin{equation}\label{eq:rho'ida}
\rho'(\id_{(0,1]} \otimes a) = (\eta' \circ \psi)(\id_{(0,1]} \otimes a),\quad a \in A.\footnote{This does \emph{not} imply that $\rho' = \eta' \circ \psi$ since $\eta' \circ
\psi$ is not necessarily a $^*$-homomorphism.}
\end{equation}
Let $\iota \colon C_0((0,1]) \to \mathcal O_\infty$ be the embedding given by $\iota(f)=f(h)$, and
let $m^\ast \colon C_0((0,1]) \to C_0((0,1]) \otimes C_0((0,1])$ be the $^*$-homomorphism induced
by multiplication. Define $\rho_0$ as the composition
\begin{equation}
C_0((0,1]) \otimes A \xrightarrow{m^\ast \otimes \id_A} C_0((0,1]) \otimes C_0((0,1]) \otimes A \xrightarrow{\iota \otimes \rho'} \mathcal O_\infty \otimes B_{\seq} \xrightarrow{\kappa_0} B_{\seq}.
\end{equation}
As $m^\ast(\id_{(0,1]}) = \id_{(0,1]} \otimes \id_{(0,1]}$ and $\iota(\id_{(0,1]}) = h$, for each
$a \in A$,
\begin{align}
\rho_0(\id_{(0,1]} \otimes a)
    &= (\kappa_0 \circ (\iota \otimes \rho'))( \id_{(0,1]} \otimes \id_{(0,1]} \otimes a) \nonumber\\
    &= \kappa_0 (h \otimes \eta'(\psi(\id_{(0,1]}\otimes a)))\quad\text{by~\eqref{eq:rho'ida}}\nonumber \\
    &= (\eta \circ \psi )(\id_{(0,1]}\otimes a).
\end{align}
Hence $\rho_0$ is the $^*$-homomorphism induced by the cpc order zero map $(\eta \circ
\psi)(\id_{(0,1]}\otimes \cdot)$ (see Proposition~\ref{prop.Cone}), and thus $\rho = \rho_0$.

The $^*$-homomorphism
\begin{equation}
\iota \otimes \rho' \colon C_0((0,1]) \otimes C_0((0,1]) \otimes A \to \mathcal O_\infty \otimes B_{\seq}
\end{equation}
is $\mathcal O_2$-stable by Lemmas \ref{Lm:HO2}~and~\ref{Lm:TensorO2}. As $\rho$ factors through
this $\mathcal O_2$-stable map by construction of $\rho_0$, it follows that $\rho$ is $\mathcal
O_2$-stable by \cite[Lemma~3.20]{Gabe17}. Hence $(d)$ is confirmed, and it remains to check~$(c)$.

For~$(c)$, take non-zero ideals $I,K \in \mathcal I(C_0((0,1]))$ and $J\in \mathcal I(A)$. Since
$\mathcal O_\infty$ is simple and since $\rho'$ satisfies Lemma~\ref{lm:Main}$(c)$,
\begin{align}
\mathcal I(\iota \otimes \rho')(K \otimes I \otimes J)
    &= \mathcal I(\iota)(K) \otimes \mathcal I(\rho')(I\otimes J)\nonumber\\
    &= \mathcal O_\infty \otimes \overline{B_{\seq} \Theta(J) B_{\seq}}.
\end{align}

As $\mathcal I(\iota \otimes \rho')$ is a $\Cu$-morphism, it preserves suprema. Since ideals of
the form $K\otimes I$ form a basis for $\mathcal I(C_0((0,1])\otimes C_0((0,1]))$, it follows that
for any non-zero ideals $I_0 \in \mathcal I(C_0((0,1])\otimes C_0((0,1]))$ and $J \in \mathcal
I(A)$,
\begin{equation}\label{eq:iotarhoI0}
\mathcal I(\iota \otimes \rho')(I_0 \otimes J) =  \mathcal O_\infty \otimes \overline{B_{\seq} \Theta(J) B_{\seq}}.
\end{equation}

By \cite[Proposition~2.15]{Gabe17}, $\mathcal I((\iota \otimes \rho')\circ (m^\ast \otimes
\id_{A})) = \mathcal I(\iota\otimes \rho') \circ \mathcal I(m^\ast \otimes \id_{A})$. Therefore,
for any non-zero ideals $I\in \mathcal I(C_0((0,1]))$ and $J \in \mathcal I(A)$,
using~\eqref{eq:iotarhoI0} at~\eqref{eq:OinftyTheta}, we have
\begin{align}
\mathcal I((\iota \otimes \rho')\circ (m^\ast \otimes \id_{A}))(I \otimes J)
    &= \mathcal I(\iota \otimes \rho')(\mathcal I(m^\ast)(I) \otimes J) \nonumber \\
    &= \mathcal O_\infty \otimes \overline{B_{\seq} \Theta(J) B_{\seq}}. \label{eq:OinftyTheta}
\end{align}
By Lemma~\ref{lem:fullelement}, the ideal $\Theta(J)$ contains a full element, say $b_J$. It
follows that $\mathcal I((\iota \otimes \rho')\circ (m^\ast \otimes \id_{A}))(I \otimes J)$ is the
ideal in $\mathcal O_\infty \otimes B_{\seq}$ generated by
\begin{equation}
1_{\mathcal O_\infty} \otimes b_J \in \mathcal O_\infty \otimes \Theta(J) \subseteq \mathcal O_\infty \otimes B_{\seq}.
\end{equation}
Hence $\mathcal I(\kappa_0)(\mathcal O_\infty \otimes \overline{B_{\seq} \Theta(J)
B_{\seq}})$ is the ideal generated by
\begin{equation}
\kappa_0(1_{\mathcal O_\infty} \otimes( b_J)_{n=1}^\infty) = (\kappa (1_{\mathcal O_\infty} \otimes b_J))_{n=1}^\infty\label{eq:kappa01bJ}.
\end{equation}
Since $\kappa (1_{\mathcal O_\infty} \otimes b_J)$ is a full element in $\Theta(J)$ by choice of
$\kappa$, the ideal generated by the element~\eqref{eq:kappa01bJ} is precisely
$\overline{B_{\seq}\Theta(J)B_{\seq}}$. Hence, for any nonzero $I \in
\mathcal{I}(C_0((0,1]))$ and $J \in \mathcal{I}(A)$,
\begin{align}
\mathcal I(\rho) (I\otimes J) &= \mathcal I(\kappa_0) \circ \mathcal I ((\iota \otimes \rho')\circ (m^\ast \otimes \id_{A}))(I \otimes J) \nonumber\\
    &=  \mathcal I(\kappa_0) (\mathcal O_\infty \otimes \overline{B_{\seq} \Theta(J) B_{\seq}})\qquad\text{by~\eqref{eq:OinftyTheta}} \nonumber\\
    &= \overline{B_{\seq} \Theta(J) B_{\seq}},
\end{align}
as required.
\end{proof}

We now turn to realising the ideal-lattice behaviour of $\mathcal O_\infty$-stable morphisms as
opposed to morphisms into $\mathcal O_\infty$-stable algebras.  We make use of
$(B_{\seq})_{\seq}$ --- that is, the sequence algebra of the sequence algebra
$B_{\seq}$ --- as a technical device. We eliminate it in the following section.

\begin{lemma}\label{l:Oinftymapmain}
Let $A$ and $B$ be $\Cstar$-algebras with $A$ separable, and suppose that $\theta \colon A \to B$
is an $\mathcal O_\infty$-stable $^*$-homomorphism. Then for $n\in\mathbb N$, there are finite dimensional
$\Cstar$-algebras $F_n$ and cpc maps $\psi_n \colon C_0((0,1]) \otimes A \to F_n$ and $\eta_n
\colon F_n \to B_{\seq}$ such that, writing
\begin{align}
\psi &\coloneqq (\psi_n)_{n=1}^\infty  \colon C_0((0,1]) \otimes A \to \frac{\prod_{n=1}^\infty  F_n}{\bigoplus_{n=1}^\infty  F_n}\quad\text{ and}\\
\eta &\coloneqq (\eta_n)_{n=1}^\infty  \colon \frac{\prod_{n=1}^\infty  F_n}{\bigoplus_{n=1}^\infty  F_n} \to (B_{\seq})_{\seq}
\end{align}
for the induced maps, the following are satisfied:
\begin{itemize}
\item[$(a)$] $\psi$ is a $^*$-homomorphism;
\item[$(b)$] each $\eta_n$ is order zero;
\item[$(c)$] if $\rho \colon C_0((0,1]) \otimes A \to (B_{\seq})_{\seq}$ is the
    $^*$-homomorphism induced by the cpc order zero map $(\eta \circ \psi)(\id_{(0,1]} \otimes
    \cdot)$ (see Proposition~\ref{prop.Cone}), then
\begin{equation}
\mathcal I(\rho)(I \otimes J) = \overline{(B_{\seq})_{\seq} \theta(J) (B_{\seq})_{\seq}}
\end{equation}
for any $J\in \mathcal I(A)$ and any non-zero $I \in \mathcal I(C_0((0,1]))$; and
\item[$(d)$] the $^*$-homomorphism $\rho$ is $\mathcal O_2$-stable.
\end{itemize}
If every quotient of $A$ is quasidiagonal, we may additionally arrange that
$\psi_n(C_0((0,1))\otimes A) = 0$ for each $n\in \mathbb N$.
\end{lemma}
\begin{proof}
As $\theta \colon A \to B$ is $\mathcal O_\infty$-stable, there is a unital embedding
\begin{equation}
j \colon \mathcal O_\infty \to \frac{B_{\seq} \cap \theta(A)'}{\Ann (\theta(A))}.
\end{equation}
Thus there is an induced $^*$-homomorphism $\theta_0 \colon A \otimes \mathcal O_\infty \to
B_{\seq}$ such that for all $a \in A$ and $x \in \mathcal O_\infty$, for any lift
$\overline{j(x)} \in B_{\seq} \cap \Theta(A)'$ of $j(x)$, we have\footnote{Note that the definition of the annihilator ensures that the expression~\eqref{e4.28} does not depend on the choice of lift $\overline{j(x)}$.}
\begin{equation}\label{e4.28}
\theta_0(a\otimes x)=\theta(a)\overline{j(x)}.
\end{equation}
In particular, $\theta_0(a \otimes 1_{\mathcal O_\infty}) = \theta(a)$ for all $a\in A$.

Let $D\coloneqq  \theta_0(A \otimes \mathcal O_\infty)$. Then $D$ is $\mathcal O_\infty$-stable
by, for example, \cite[Corollary~3.3]{TW:Trans}, and $\theta$ corestricts to a $^*$-homomorphism
$\theta|^D \colon A \to D$. Let $\Theta \coloneqq  \mathcal I(\theta|^D) \colon \mathcal I(A) \to
\mathcal I(D)$, be the $\Cu$-morphism induced by $\theta|^D$. Apply Lemma~\ref{l:O2stablelift} to
obtain $(F_n, \psi_n, \eta'_n)_{n=1}^\infty$ and $\rho' \colon C_0((0,1]) \otimes A \to
D_{\seq}$ satisfying $(a)$--$(d)$ of that lemma and satisfying $\psi_n(C_0((0,1]) \otimes A) =
0$ for all $n$ if every quotient of $A$ is quasidiagonal. Let $\iota \colon D \hookrightarrow
B_{\seq}$ be the inclusion, and $\eta_n \coloneqq  \iota \circ \eta_n' \colon F_n \to
B_{\seq}$. We will show that $(F_n, \psi_n, \eta_n)_{n=1}^\infty$ satisfy conditions
$(a)$--$(d)$.

Clearly~$(a)$ is satisfied, and also the final assertion of the lemma. As each $\eta_n'$ is order
zero, so too is each $\eta_n$, so~$(b)$ is satisfied.

Let $\iota_{\seq} \colon D_{\seq} \to (B_{\seq})_{\seq}$ be the $^*$-homomorphism
on sequence algebras induced by $\iota$. Clearly $\iota_{\seq} \circ \eta' = \eta$ where
$\eta' = (\eta_n')_{n=1}^\infty$. So for each $a \in A$,
\begin{equation}
(\iota_{\seq} \circ \rho') (\id_{(0,1]} \otimes a) = (\iota_{\seq} \circ \eta' \circ \psi)(\id_{(0,1]} \otimes a) = (\eta \circ \psi)(\id_{(0,1]} \otimes a).
\end{equation}
Hence $\iota_{\seq} \circ \rho'$ is the unique $^*$-homomorphism induced by the cpc order zero
map $(\eta \circ \psi)(\id_{(0,1]} \otimes \cdot)$ (see Proposition~\ref{prop.Cone}), so $\rho =
\iota_{\seq} \circ \rho'$ by uniqueness. As the composition of an $\mathcal O_2$-stable
$^*$-homomorphism with any $^*$-homomorphism is again $\mathcal O_2$-stable by
\cite[Lemma~3.20]{Gabe17}, the $\mathcal O_2$-stability of $\rho'$ yields $\mathcal O_2$-stability
of $\rho$. Hence~$(d)$ is satisfied and it remains to check~$(c)$.

To show $(c)$, note that, writing $\iota_B \colon B \to B_{\seq}$ for the canonical embedding,
\begin{equation}\label{eq:iotatheta}
\iota \circ \theta|^D = \iota_B \circ \theta.
\end{equation}
Functoriality of $\mathcal I$ gives
\begin{equation}
\mathcal I(\iota) \circ \Theta = \mathcal I(\iota \circ \theta|^D) = \mathcal I(\iota_B) \circ \mathcal I(\theta).
\end{equation}
Let $\iota_D \colon D \to D_{\seq}$ be the canonical embedding. Fix non-zero ideals $I \in
\mathcal I(C_0((0,1]))$ and $J \in \mathcal I(A)$. Using functoriality of $\mathcal{I}$ at the
first step and that $\rho$ satisfies Lemma~\ref{l:O2stablelift}$(c)$ at the second, we calculate:
\begin{align}
\mathcal I(\rho) (I\otimes J) &= \mathcal I(\iota_{\seq}) \circ \mathcal I(\rho') (I\otimes J) \nonumber\\
    &= \mathcal I(\iota_{\seq}) (\overline{D_{\seq} \Theta(J) D_{\seq}}) \label{eq:Oinftymapmain.e1}\nonumber\\
    &= \mathcal I(\iota_{\seq}) \circ \mathcal I(\iota_D) \circ \mathcal I(\theta|^D)(J).
\end{align}
As $\iota_{\seq} \circ \iota_D \circ \theta|^D$ is equal to the composition of $\theta$ with
the canonical inclusion $B \to (B_{\seq})_{\seq}$, it follows that
\begin{equation}
 \mathcal I(\iota_{\seq}) \circ \mathcal I(\iota_D) \circ \mathcal I(\theta|^D)(J) = \overline{(B_{\seq})_{\seq} \theta(J) (B_{\seq})_{\seq}}
\end{equation}
which verifies~$(c)$.
\end{proof}


\section{Nuclear dimension of $\mathcal O_\infty$-stable maps}\label{S5}

We now turn to the main results of the paper, combining the existence results of the previous
sections with classification theorems to compute the nuclear dimension of $\mathcal
O_\infty$-stable maps. The appropriate notion of equivalence in this context is approximate
Murray--von Neumann equivalence:

\begin{definition}
Let $A$ and $B$ be $\Cstar$-algebras with $A$ separable, and let $\phi,\psi \colon A\to B$ be
$^*$-homo\-morphisms. We say that $\phi$ and $\psi$ are \emph{approximately Murray--von Neumann
equivalent} if there exists $v\in B_{\seq}$ such that  $v^* \phi(a)v=\psi(a)$ and
$v\psi(a)v^*=\phi(a)$ for all $a\in A$.
\end{definition}

The precise classification ingredient we need is the following uniqueness theorem in the spirit of
Kirchberg's classification results \cite{K:Book,K:German}.

\begin{theorem}[{\cite[Theorem~3.23]{Gabe17}}]\label{Thm:O2MvN}
Let $A$ and $B$ be $\Cstar$-algebras with $A$ separable and exact. Suppose that $\phi,\psi \colon
A\to B$ are nuclear, $\mathcal O_2$-stable $^*$-homo\-morphisms. Then $\phi$ and $\psi$ are
approximately Murray--von Neumann equivalent if and only if $\mathcal I(\phi)=\mathcal I(\psi)$.
\end{theorem}

We can now prove Theorem~\ref{ThmC}, which asserts that every $\mathcal O_\infty$-stable
homomorphism $\theta$ out of a separable exact $\Cstar$-algebra $A$ has nuclear dimension at most~1,
and that if moreover every quotient of $A$ is quasidiagonal, then $\theta$ also has decomposition rank at
most~1.

\begin{proof}[Proof of Theorem~\ref{ThmC}]
Let $A$ and $B$ be $\Cstar$-algebras with $A$ separable and exact and let $\theta\colon A\rightarrow B$ be a nuclear $\mathcal O_\infty$-stable $^*$-homomorphism. Fix a positive element $h_0\in \mathcal O_\infty$ with spectrum $[0,1]$, and let $h_1 =
1_{\mathcal O_\infty}- h_0$. For $i \in \{0,1\}$, we write $\iota_i \colon C_0((0,1]) \to \mathcal
O_\infty$ for the injective $^*$-homo\-morphism that sends the generator $\id_{(0,1]}\in
C_0((0,1])$ to $h_i$. These $\iota_i$ are $\mathcal O_2$-stable by Lemma~\ref{Lm:HO2}, and
Lemma~\ref{Lm:TensorO2} implies that the $^*$-homomorphisms $\iota_i \otimes \id_A \colon
C_0((0,1]) \otimes A \to \mathcal O_\infty \otimes A$ are also $\mathcal O_2$-stable.

As $\theta$ is $\mathcal O_\infty$-stable, there exists a unital embedding
\begin{equation}\label{eq:jOinftyrelcom}
j \colon \mathcal O_\infty  \to \frac{B_{\seq} \cap \theta(A)'}{ \Ann \theta(A)}.
\end{equation}
Thus there exists a $^*$-homomorphism $j\times \theta\colon \mathcal O_\infty\otimes A\rightarrow
B_{\seq}$ such that for $x \in \mathcal O_\infty$ and $a \in A$, for any lift $\overline{j(x)}
\in B_{\seq} \cap \theta(A)'$ of $j(x)$, we have $(j\times\theta)(x\otimes
a)=\overline{j(x)}\theta(a)$. For $i \in \{0,1\}$, let $\mu^{(i)}$ denote the composition
\begin{equation}\label{eq:muidef}
C_0((0,1]) \otimes A \xrightarrow{\iota_i \otimes \id_A} \mathcal O_\infty \otimes A \xrightarrow{j \times \theta} B_{\seq} \subseteq (B_{\seq})_{\seq}.
\end{equation}
Since the $\iota_i \otimes \id_A$ are $\mathcal O_2$-stable, \cite[Lemma~3.20]{Gabe17} implies
that the compositions $\mu^{(i)} \coloneqq  (j \times \theta) \circ (\iota_i \otimes \id_A)$ are
$\mathcal O_2$-stable. To see that each $\mu^{(i)}$ is nuclear, it suffices by
\cite[Theorem~2.9]{Gabe-JFA} to show that the order zero map
\begin{equation}\label{eq:thetaidef}
\theta^{(i)} \coloneqq  \mu^{(i)} (\id_{(0,1]} \otimes \cdot) \colon A \to (B_{\seq})_{\seq}
\end{equation}
is nuclear. Let $\overline h_i \in B_{\seq} \cap \theta(A)' \subseteq
(B_{\seq})_{\seq} \cap \theta(A)'$ be a positive lift of $j(h_i) = j(\iota_i
(\id_{(0,1]}))$. Then
\begin{equation}\label{eq:phiirhoi}
\theta^{(i)} (a) = \mu^{(i)} (\id_{(0,1]} \otimes a) = \overline h_i \theta(a) = \overline h_i^{1/2} \theta(a) \overline h_i^{1/2},\quad a\in A.
\end{equation}
As $\theta^{(i)}$ is the composition of the nuclear map $\theta \colon A \to B \subseteq
(B_{\seq})_{\seq}$ and the cp~map $\overline h_i^{1/2} (\cdot) \overline h_i^{1/2} \colon
(B_{\seq})_{\seq} \to (B_{\seq})_{\seq}$, it is nuclear. Thus, each $\mu^{(i)}$ is
nuclear.

Since $h_0 + h_1 = 1_{\mathcal O_\infty}$, we have $1_{B_{\seq}} - (\overline h_0 + \overline
h_1) \in \Ann \theta(A)$. Hence~\eqref{eq:phiirhoi} yields
\begin{equation}
\theta^{(0)}(a) + \theta^{(1)}(a) =
    \theta(a)(\overline h_0 + \overline h_1) = \theta(a),\quad a \in A.
\end{equation}
Thus
\begin{equation}\label{eq:theta0theta1}
\theta^{(0)} + \theta^{(1)} = \theta \colon A \to (B_{\seq})_{\seq}.
\end{equation}

Let $(F_n, \psi_n, \eta_n)_{n=1}^\infty$ and $\rho \colon C_0((0,1])\otimes A \to
(B_{\seq})_{\seq}$ be as provided by Lemma~\ref{l:Oinftymapmain} for the $\mathcal
O_\infty$-stable map $\theta$. If all quotients of $A$ are quasidiagonal, we choose $\psi_n$ such
that $\psi_n(C_0((0,1))\otimes A) = 0$. By construction, the order zero map $\rho(\id_{(0,1]}
\otimes \cdot ) \colon A \to (B_{\seq})_{\seq}$ is represented by the sequence
\begin{equation}\label{5.7}
(\eta_n \circ \psi_n(\id_{(0,1]} \otimes \cdot) \colon A \to B_{\seq})_{n=1}^\infty.
\end{equation}
As each $\eta_n \circ \psi_n (\id_{(0,1]} \otimes \cdot)$ is nuclear (they factor through the
finite dimensional $\Cstar$-algebra $F_n$) and $A$ is exact, $\rho(\id_{(0,1]} \otimes \cdot)$ is
nuclear by \cite[Proposition~3.3]{D:JFA}. So \cite[Theorem~2.9]{Gabe-JFA} implies that $\rho$ is
nuclear.

We will show that $\mathcal I(\rho) = \mathcal I(\mu^{(i)})$ for $i=0,1$. For this, fix $i \in
\{0,1\}$, and fix non-zero $I \in \mathcal{I}(C_0((0,1]))$ and $J \in \mathcal{I}(A)$. As $\rho$
and $\mu^{(i)}$ are  $\Cu$-morphisms and thus preserve suprema, and as $\{I \otimes J : I \in
\mathcal{I}(C_0((0,1]))\text{ and }J \in \mathcal{I}(A)\}$ is a basis for $\mathcal
I(C_0((0,1])\otimes A)$ it suffices to check that $\mathcal I(\rho)(I\otimes J) = \mathcal
I(\mu^{(i)})(I\otimes J)$.

By Lemma~\ref{l:Oinftymapmain}$(c)$, we have
\begin{equation}\label{eq:IrhoIJ}
\mathcal I(\rho)(I \otimes J) = \overline{(B_{\seq})_{\seq} \theta(J) (B_{\seq})_{\seq}}.
\end{equation}
As $I$ is non-zero, and as $\mathcal I(\cdot)$ is a functor (\cite[Proposition~2.5]{Gabe17}), the
definition~\eqref{eq:muidef} of $\mu^{(i)}$ gives
\begin{align}
\mathcal I(\mu^{(i)}) (I \otimes J)
    &= \mathcal I(j \times \theta)(\mathcal I (\iota_i \otimes \id_A)(I \otimes J))\nonumber\\
    &= \mathcal I(j \times \theta)(\mathcal O_\infty \otimes J).\label{eq:Imu(i)}
\end{align}
By the definition of $j \times \theta \colon \mathcal O_\infty \otimes A \to B_{\seq} \subseteq
(B_{\seq})_{\seq}$, the ideal $(j \times \theta)(\mathcal O_\infty \otimes J)$ is
contained in the ideal generated by $\theta(J)$. Conversely, the defining property of $j \times
\theta$, and then~\eqref{eq:Imu(i)} give
\begin{equation}
\theta(J) = (j \times \theta)(1_{\mathcal O_\infty} \otimes J) \subseteq \mathcal I(\mu^{(i)})(I \otimes J).
\end{equation}
Hence $\mathcal I(\mu^{(i)})(I \otimes J)$ is the ideal generated by $\theta(J)$.
Thus~\eqref{eq:IrhoIJ} gives $\mathcal{I}(\mu^{(i)})(I \otimes J) = \mathcal{I}(\rho)(I \otimes
J)$ as required.

Since $\rho$ and $\mu^{(i)}$ are nuclear and $\mathcal O_2$-stable, Theorem~\ref{Thm:O2MvN}
implies that $\rho$ and $\mu^{(i)}$ are approximately Murray--von Neumann equivalent. Hence
\cite[Proposition~3.10]{Gabe17} shows that
\begin{equation}
\rho\oplus 0\text{ and } \mu^{(i)} \oplus 0 \colon C_0((0,1]) \otimes A \to M_2((B_{\seq})_{\seq})
\end{equation}
are approximately unitary equivalent (via unitaries in the minimal unitisation, though we work in the slightly larger unitisation $M_2((\widetilde B_{\seq})_{\seq})$). As $A$ is
separable, a standard reindexing argument for sequence algebras shows that $\rho\oplus 0$ and
$\mu^{(i)} \oplus 0$ are unitary equivalent. Choose unitaries
\begin{equation}
u^{(i)} \in M_2((\widetilde B_{\seq})_{\seq}) = M_2(\widetilde B_{\seq})_{\seq}
\end{equation}
such that
\begin{equation}\label{eq:uirhomu}
u^{(i)\ast} (\rho(x)\oplus 0) u^{(i)} = \mu^{(i)}(x) \oplus 0,\quad x\in C_0((0,1])\otimes A,\ i=0,1.
\end{equation}
Choose a unitary lift $(u^{(i)}_n)_{n=1}^\infty \in \prod_{n=1}^\infty M_2(\widetilde
B_{\seq})$ of each $u^{(i)}$.

Define cpc maps
\begin{equation}
\widehat \psi_n \coloneqq  \psi_n(\id_{(0,1]} \otimes \cdot) \colon A \to F_n,
\end{equation}
and, for $i=0,1$,
\begin{equation}
\widehat \eta_n^{(i)}\coloneqq  u^{(i)\ast}_n (\eta_n(\cdot)\oplus 0) u^{(i)}_n \colon F_n \to M_2(B_{\seq}).
\end{equation}
Then $\widehat{\psi}_n\oplus\widehat{\psi}_n\colon A\rightarrow F_n\oplus F_n$ is cpc for each $n$, and,
as each $\eta_n$ is cpc order zero by Lemma~\ref{l:Oinftymapmain}$(b)$, each
$\widehat{\eta}_n^{(i)}$ is cpc order zero.  Computing in $M_2(\tilde{B}_{\seq})_{\seq}$, for each
$a \in A$, we have
\begin{align}
\theta(a) \oplus 0
    &= (\theta^{(0)}(a) \oplus 0) + (\theta^{(1)}(a) \oplus 0) &&\text{by~\eqref{eq:theta0theta1}}\nonumber\\
    &= (\mu^{(0)}(\id_{(0,1]} \otimes a) \oplus 0) + (\mu^{(1)}(\id_{(0,1]}\otimes a) \oplus 0)&&\text{by~\eqref{eq:thetaidef}}   \nonumber\\
    &= \sum_{i=0}^1 u^{(i)\ast} (\rho(\id_{(0,1]} \otimes a)\oplus 0) u^{(i)} &&\text{by~\eqref{eq:uirhomu}} \nonumber\\
    &= \left( \sum_{i=0}^1 u^{(i)\ast}_n  ((\eta_n (\psi_n(\id_{(0,1]} \otimes a)))\oplus 0) u^{(i)}_n \right)_{n=1}^\infty\nonumber&&\text{by~\eqref{5.7}} \\
    &= \left(  \sum_{i=0}^1 (\widehat \eta_n^{(i)} \circ \widehat\psi_n)(a) \right)_{n=1}^\infty  \nonumber\\
    &= \left( ((\widehat \eta_n^{(0)} + \widehat\eta_n^{(1)}) \circ ( \widehat\psi_n \oplus \widehat\psi_n))(a)\right)_{n=1}^\infty.
\end{align}
Thus $\big(F_n\oplus
F_n,\widehat{\psi}_n\oplus\widehat{\psi}_n,\widehat{\eta}^{(0)}+\widehat{\eta}^{(1)}\big)_{n=1}^\infty$
is a sequence of $1$-decomposable approximations witnessing that $\theta\oplus 0\colon A\rightarrow
M_2(B_{\seq})$ has nuclear dimension at most $1$.  By Proposition~\ref{Prop:cpcHer} it follows
that $\theta \colon A \to B_{\seq}$ also has nuclear dimension at most $1$. Hence
$\dimnuc(\theta)\leq 1$ by \cite[Proposition~2.5]{TW:APDE}.

If all quotients of $A$ are quasidiagonal, then we can additionally choose the $\psi_n$ in
Lemma~\ref{l:Oinftymapmain} so that $\psi_n(C_0((0,1))\otimes A)=0$.  Then
Lemma~\ref{l:Oinftymapmain}$(a)$ implies that
\begin{equation}
( \widehat \psi_n \oplus \widehat \psi_n)_{n=1}^\infty \colon A \to \frac{\prod_{n=1}^\infty  (F_n \oplus F_n)}{\bigoplus_{n=1}^\infty(F_n \oplus F_n)}
\end{equation}
is a $^*$-homomorphism. Now Lemma~\ref{detectdr} shows that $\theta\oplus 0\colon A\rightarrow
M_2(B_{\seq})$ has decomposition rank at most $1$. Cutting to the hereditary subalgebra
$B_{\seq}$ by Proposition~\ref{Prop:cpcHer}, and removing the sequence algebra by
\cite[Proposition~2.5]{TW:APDE} just as above, we obtain $\dr\,\theta \leq 1$.
\end{proof}

Theorem~\ref{ThmA}, which says that separable, nuclear, $\mathcal O_\infty$-stable
$\Cstar$-algebras have nuclear dimension~1 is an immediate consequence of Theorem~\ref{ThmC}.

\begin{proof}[Proof of Theorem~\ref{ThmA}]
Let $A$ be a separable, nuclear, $\mathcal O_\infty$-stable $\Cstar$-algebra. By
\cite[Proposition~3.19]{Gabe17}, $\id_A$ is $\mathcal O_\infty$-stable, so
\begin{equation}
\dimnuc A = \dimnuc \id_A \leq 1
\end{equation}
by Theorem~\ref{ThmC}. As $\mathcal O_\infty$-stable $\Cstar$-algebras are not approximately finite dimensional, it follows from
\cite[Theorem~3.4]{Winter-JFA} that $\dimnuc A >0$ so $\dimnuc A = 1$.
\end{proof}

\section{Finite decomposition rank}\label{S6}

In \cite{Rordam-IJM}, R{\o}rdam constructs a separable, nuclear, $\mathcal O_\infty$-stable
$\Cstar$-algebra $\mathcal A_{[0,1]}$ that is an inductive limit of $\Cstar$-algebras of the form
$C_0((0,1]) \otimes M_{2^n}$. This provides an example of a separable, nuclear, $\mathcal
O_\infty$-stable $\Cstar$-algebra with decomposition rank one. In this section we characterise
exactly when separable, nuclear, $\mathcal O_\infty$-stable $\Cstar$-algebras have finite
decomposition rank.

In \cite{Gabe:new} it was established that a separable, nuclear, $\mathcal O_\infty$-stable
$\Cstar$-algebra is quasidiagonal if and only if its primitive ideal space\footnote{Recall that
the \emph{primitive ideal space} $\Prim A$ (sometimes denoted $\check A$) of a $\Cstar$-algebra
$A$ is the set of all primitive ideals equipped with the Jacobson topology, see
\cite[Section~4.1]{P:Book}. The primitive ideal space is often non-Hausdorff, but is always $T_0$
\cite[4.1.4]{P:Book}.} has no non-empty, compact, open subsets.\footnote{More generally, this
holds for any separable, exact $\Cstar$-algebra that is traceless in the sense of \cite{KR-adv}.
That $\mathcal O_\infty$-stable $\Cstar$-algebras are traceless follows from
\cite[Theorem~9.1]{KR-adv}.} This will be used to give a characterisation of when separable,
nuclear, $\mathcal O_\infty$-stable $\Cstar$-algebras have finite decomposition rank in terms of
their primitive ideal space.

Recall that a subset $C$ of a topological space $X$ is called \emph{locally closed} if there are
open subsets $U \subseteq V \subseteq X$ such that $C = V \setminus U$. Equivalently, $C$ is
locally closed if and only if there is an open subset $V \subseteq X$ such that $C = V \cap
\overline C$, where $\overline{C}$ denotes the closure of $C$.

\begin{remark}\label{r:locallyclosed}
Recall from \cite[Theorem~4.1.3]{P:Book} that there is an order isomorphism from $\mathcal I(A)$
to the collection of open subsets of $\Prim A$ that carries $I \in \mathcal{I}(A)$ to $U_I
\coloneqq  \{J \in \Prim A : I \not\subseteq J\}$. Since $I \mapsto U_I$ is an isomorphism of
lattices, we have $U_I \cap U_J = U_{I\cap J}$ and $U_I \cup U_J = U_{I+J}$ for $I, J \in
\mathcal{I}(A)$.

By a subquotient of a $\Cstar$-algebra $A$, we mean the quotient of an ideal in $A$; that is, a
$\Cstar$-algebra of the form $J/I$ where $I \subseteq J$ are ideals in $A$. The primitive ideal
space $\Prim(J/I)$ is homeomorphic to the locally closed subset $U_J \setminus U_I$ of $\Prim A$
by \cite[Theorem~4.1.11(ii)]{P:Book}. If $I_i \subseteq J_i$ are ideals in $A$ for $i=1,2$ such
that $U_{J_1} \setminus U_{I_1} = U_{J_2} \setminus U_{I_2}$, then the inclusions $I_i
\hookrightarrow I_1+ I_2$ and $J_i \hookrightarrow J_1 + J_2$ induce isomorphisms $J_1/I_1 \cong
(J_1+J_2)/(I_1+I_2) \cong J_2 /I_2$.\footnote{In fact, since $U_{J_1} \cup U_{J_2} = U_{J_1} \cup
(U_{J_2}\setminus U_{I_2}) \cup U_{I_2} = U_{J_1} \cup (U_{J_1} \setminus U_{I_1}) \cup U_{I_2} =
U_{J_1} \cup U_{I_2}$, one has $J_1 + J_2 = J_1 + I_2$. Similarly, $(U_{I_1} \cup U_{I_2}) \cap
U_{J_1} = U_{I_1} \cup (U_{I_2} \cap (U_{J_1} \setminus U_{I_1})) = U_{I_1} \cup (U_{I_2} \cap
(U_{J_2} \setminus U_{I_2})) = U_{I_1}$, and thus $I_1 = (I_1 + I_2) \cap J_1$. Hence
\[
\frac{J_1}{I_1} = \frac{J_1}{(I_1 + I_2) \cap J_1} \cong \frac{J_1 + I_1 + I_2}{I_1 + I_2} = \frac{J_1 + J_2}{I_1+I_2}
\]
where the isomorphism is induced by inclusions. By symmetry the result follows.} It follows that
there is an equivalence relation on subquotients of $A$ given by $J_1/I_1 \sim J_2/I_2$ if and
only if the inclusions $I_i \hookrightarrow I_1+ I_2$ and $J_i \hookrightarrow J_1 + J_2$ induce
isomorphisms $J_1/I_1 \cong (J_1+J_2)/(I_1+I_2) \cong J_2 /I_2$. It also follows that the map $J/I
\mapsto U_J \setminus U_I$ descends to a bijection between equivalence classes of subquotients of
$A$ and locally closed subsets of $\Prim A$.
\end{remark}

By \cite[Proposition~4.1.1]{R-book} a simple $\Cstar$-algebra $A$ is purely infinite if and only if every non-zero hereditary
$\Cstar$-subalgebra contains a non-zero, $\sigma$-unital, stable $\Cstar$-subalgebra. The
following is a similar characterisation.

\begin{proposition}\label{p:herstable}
Let $A$ be a $\Cstar$-algebra. Then every non-zero, $\sigma$-unital hereditary
$\Cstar$-sub\-algebra of $A$ is stable if and only if $A$ is purely infinite and $\Prim A$ has no
locally closed, one-point subsets.
\end{proposition}
\begin{proof}
Suppose that $A$ is purely infinite and that $\Prim A$ has no locally closed, one-point subsets,
and let $B \subseteq A$ be a non-zero, $\sigma$-unital hereditary $\Cstar$-subalgebra. Then $B$ is
also purely infinite by \cite[Proposition~4.17]{KR-AJM}. By \cite[Theorem~4.24]{KR-AJM} $B$ is
stable if (and only if) $B$ has no non-zero, unital quotients. Suppose for contradiction that $B$
has a unital quotient. As any unital $\Cstar$-algebra has a simple quotient, $B$ has a simple
quotient, say $B/I_B$. Let $I_A \coloneqq \overline{AI_BA}$ and $J_A \coloneqq \overline{ABA}$ be
the ideals in $A$ generated by $I_B$ and $B$ respectively. As $B$ is a full, hereditary
$\Cstar$-subalgebra of $J_A$, it follows that $B/I_B \cong (B+I_A)/I_A$ is a full, hereditary
$\Cstar$-subalgebra of $J_A/I_A$. Then $B/I_B$ and $J_A/I_A$ are strongly Morita
equivalent,\footnote{If $A\subseteq B$ is a full, hereditary $\Cstar$-subalgebra, then
$\overline{AB}$ is an imprimitivity $A$--$B$-bimodule, so $A$ and $B$ are strongly Morita
equivalent. See \cite[Chapter~3]{RW:book} for more details.} and thus $J_A/I_A$ is simple by the
Rieffel correspondence.\footnote{The Rieffel correspondence asserts that strongly Morita
equivalent $\Cstar$-algebras have order isomorphic ideal lattices, see for instance
\cite[Theorem~3.22]{RW:book}. In particular, if one is simple then so is the other.} Hence $\Prim
A$ has a locally closed, one-point set, namely $U_{J_A} \setminus U_{I_A}$, by
Remark~\ref{r:locallyclosed}. This is a contradiction, and so $B$ has no unital quotients. Hence
$B$ is stable.

Conversely, suppose that all non-zero, $\sigma$-unital hereditary $\Cstar$-sub\-algebras of $A$
are stable. Then every non-zero, positive element $a\in A$ is stable in the sense that
$\overline{aAa}$ is stable. By \cite[Proposition~3.7 and Theorem~4.16]{KR-AJM} it follows that $A$
is purely infinite. Suppose for contradiction that $\Prim A$ has a locally closed, one-point
subset $\{x\}$. By Remark~\ref{r:locallyclosed} there exist ideals $I \subseteq J$ of $A$ such
that $\{x\} = U_J \setminus U_I$, and hence $J/I$ is simple. As pure infiniteness passes to ideals
and quotients by \cite[Theorem~4.19]{KR-AJM}, it follows that $J/I$ is purely infinite and simple.
Hence $J/I$ contains a non-zero projection $p$. Let $a\in J$ be a positive lift of $p$. Then
$\overline{aAa}$ has a unital quotient, namely $p(J/I) p$, contradicting stability of
$\overline{aAa}$. Hence $\Prim A$ has no locally closed, one-point subsets.
\end{proof}

We do not know if the $\Cstar$-algebras considered in Proposition~\ref{p:herstable} are actually
strongly purely infinite, and thus $\mathcal O_\infty$-stable under separability and nuclearity
assumptions by \cite{KR-adv}.

We can now prove Theorem~\ref{ThmB}, which states that for separable, nuclear, $\mathcal
O_\infty$-stable $\Cstar$-algebras $A$, the following are equivalent:
\begin{itemize}
\item[$(i)$] the decomposition rank of $A$ is finite;
\item[$(ii)$] the decomposition rank of $A$ is $1$;
\item[$(iii)$] every quotient of $A$ is quasidiagonal;
\item[$(iv)$] every non-zero, hereditary $\Cstar$-subalgebra of $A$ is stable;
\item[$(v)$] the primitive ideal space of $A$ has no locally closed, one-point subsets.
\end{itemize}

\begin{proof}[Proof of Theorem~\ref{ThmB}]
$(iii) \Rightarrow (ii)$: Theorem~\ref{ThmC} shows that when all quotients of $A$ are quasidiagonal, then $\dr\,\id_A\leq 1$ (as the identity map inherits $\mathcal O_\infty$-stability from $A$).  Accordingly $\dr\, A=1$, as $A$ is not approximately finite dimensional.

$(ii) \Rightarrow (i)$: This is obvious.

$(i)\Rightarrow (v)$: Suppose that $A$ has finite decomposition rank, and assume for
contradiction that $\Prim A$ has a locally closed, one-point subset. By
Remark~\ref{r:locallyclosed}, $A$ has a simple subquotient. As both pure infiniteness and finite
decomposition rank pass to ideals and to quotients by \cite[Theorem~4.19]{KR-AJM} and \cite[(3.3)
and Proposition~3.8]{KW:IJM}, this subquotient is purely infinite, simple and has finite
decomposition rank. This contradicts that $\Cstar$-algebras of finite decomposition rank are
quasidiagonal \cite[Proposition~5.1]{KW:IJM}, and hence stably finite (see for example
\cite[Proposition~7.1.15]{BO:Book}).

$(v) \Rightarrow (iii)$: We prove the contrapositive statement, so suppose that $I$ is an ideal
in $A$ such that $A/I$ is not quasidiagonal. Let $U_I \subseteq \Prim A$ be the open subset
corresponding to $I$. By \cite[Corollary~C]{Gabe:new}, $\Prim(A/I)$ has a non-empty, compact, open
subset $V \subseteq \Prim(A/I)$. By compactness of $V$ and Zorn's lemma, $V$ contains a maximal,
open, proper subset $W$.\footnote{Consider the collection $\mathcal U$ of proper open subsets of
$V$. Let $\mathcal C \subseteq \mathcal U$ be a chain and let $U \coloneqq \bigcup \mathcal C$
which is an open subset of $V$. As $V$ is compact it follows that $U \neq V$ (otherwise $U' = V$
for some $U' \in \mathcal C$, a contradiction), and thus $U \in \mathcal U$ is an upper bound for
$\mathcal C$. By Zorn's lemma, $\mathcal U$ has a maximal element.} As primitive ideal spaces of
$\Cstar$-algebras are $T_0$ (see \cite[4.1.4]{P:Book}), it follows that $V\setminus W$ is a
singleton.\footnote{If $x,y \in V \setminus W$ are distinct, then that $\Prim(A)$ is $T_0$ implies
the existence of an open neighbourhood $X$ of one that does not contain the other; but then $W
\subsetneq W \cup (X \cap V) \subsetneq V$, contradicting maximality of $W$.} Let $W', V'
\subseteq (\Prim A) \setminus U_I$ be the relatively open subsets corresponding to $V$ and $W$ under the
canonical identification of $\Prim (A/I)$ with $(\Prim A) \setminus U_I$ of
\cite[Theorem~4.1.11]{P:Book}. Then $W_0 \coloneqq  W' \cup U_I$ and $V_0 \coloneqq  V' \cup U_I$
are open subsets of $\Prim A$, and $V_0 \setminus W_0$ is a singleton, so $\Prim A$ has a locally
closed, one-point subset.

$(iv)\Leftrightarrow (v)$: This follows from Proposition~\ref{p:herstable} as $A$ is $\mathcal
O_\infty$-stable and thus purely infinite.
\end{proof}

In particular for separable nuclear $\mathcal O_\infty$-stable $\Cstar$-algebras, finiteness of the decomposition rank is determined by the primitive ideal space.

\begin{corollary}
Let $A$ and $B$ be separable, nuclear, $\mathcal O_\infty$-stable $\Cstar$-algebras such that
$\Prim A$ and $\Prim B$ are homeomorphic. If $A$ has finite decomposition rank then so does $B$.
\end{corollary}

\begin{example}
The $\Cstar$-algebra $\mathcal A_{[0,1]}$ of \cite{Rordam-IJM} has primitive ideal space
homeomorphic to $[0,1)$ equipped with the topology $\{\emptyset\} \cup \{[0,t) : 0 < t \le 1\}$. In this
space, a non-empty intersection of an open and a closed set has the form $(a,b)$, so the space has
no locally closed, one-point sets. Therefore, any separable, nuclear, $\mathcal O_\infty$-stable
$\Cstar$-algebra $B$ with $\Prim B \cong [0,1)$ with topology as above has decomposition rank~1.
\end{example}

A consequence of \cite[Proposition~6.2]{KR-GAFA} is that $\mathcal A_{[0,1]} \otimes B$ has
decomposition rank 1 for any non-zero, separable, nuclear $\Cstar$-algebra $B$. This will be
generalised in Corollary~\ref{c:drtensor} below. We will need the following lemma.

\begin{lemma}\label{l:1pt}
Let $X$ and $Y$ be topological spaces. If $\{(x,y)\} \subseteq X \times Y$ is a locally closed,
one-point subset, then so is $\{ x\} \subseteq X$.
\end{lemma}
\begin{proof}
Pick an open neighbourhood $U$ of $(x,y)$ such that $\{ (x,y)\} = U \cap \overline{\{(x,y)\}}$.
Let $U_X \subseteq X$ and $U_Y \subseteq Y$ be open subsets such that $(x,y) \in U_X \times U_Y
\subseteq U$.
Then $\{(x,y)\} = (U_X \times U_Y) \cap \overline{\{(x,y)\}} = (U_X \times U_Y) \cap
\big(\overline{\{x\}} \times \overline{\{y\}}\big) = \big(U_X \cap \overline{\{x\}}\big) \times
\big(U_Y \cap \overline{\{y\}}\big)$. In particular, $\{x\} = U_X \cap \overline{\{x\}}$.
\end{proof}

If a simple $\Cstar$-algebra $A$ has finite decomposition rank, then $A \otimes \mathcal O_\infty$
is purely infinite and simple, and thus has infinite decomposition rank.\footnote{If
$A\otimes\mathcal O_\infty$ has finite decomposition rank, then it is quasidiagonal by
\cite[Theorem~5.1]{KW:IJM}, and so is stably finite (see
\cite[Proposition~7.1.15]{BO:Book}, for example), a contradiction.} However, in the non-simple
case we obtain the following permanence property.

\begin{corollary}\label{c:drtensor}
Let $A$ be a separable, nuclear, $\mathcal O_\infty$-stable $\Cstar$-algebra with finite
decomposition rank. Then $A \otimes B$ has decomposition rank one for any non-zero, separable,
nuclear $\Cstar$-algebra $B$.
\end{corollary}
\begin{proof}
The tensor product $A\otimes B$ is separable, nuclear and $\mathcal O_\infty$-stable. As $A$ and
$B$ are separable and nuclear we have $\Prim (A\otimes B) \cong \Prim A \times \Prim B$ (see for
instance \cite[Theorem~IV.3.4.25]{B:book:opalg}). By Theorem~\ref{ThmB}, $\Prim A$ has no locally
closed, one-point subsets, so Lemma~\ref{l:1pt} implies that $\Prim (A \otimes B)$ has no locally
closed, one-point subsets. So Theorem~\ref{ThmB} implies that $A\otimes B$ has decomposition rank
one.
\end{proof}

\begin{problem}
Characterise when nuclear, $\mathcal O_\infty$-stable $^*$-homomorphisms have finite decomposition
rank.
\end{problem}


\section{Zero dimensional $\mathcal O_2$-stable $^*$-homomorphisms}\label{S7}

A $\Cstar$-algebra $A$ has nuclear dimension zero (equivalently decomposition rank zero) if and
only if it is AF, in the sense of being able to approximate finite dimensional subsets by finite
dimensional subalgebras; see \cite{KW:IJM}.  Here we begin the investigation of when
$^*$-homomorphisms are zero dimensional, resolving the question for full, $\mathcal O_2$-stable
maps with exact domains. Recall that fullness is a simplicity criterion for $^*$-homomorphisms:
$\theta \colon A \to B$ is called \emph{full} if the image under $\theta$ of every non-zero $a \in
A$ is full in $B$. Equivalently, $\theta$ is full if, for every $I \in \mathcal I(A)$,
\begin{equation}
\mathcal I(\theta)(I) = \left\{\begin{array}{ll}
0 & \textrm{ if $I=0$,} \\
B & \textrm{ otherwise.}
\end{array} \right.
\end{equation}

We begin with the special case of embeddings into $\mathcal O_2$, where Kirchberg famously showed that
every separable and exact $\Cstar$-algebra $A$ admits such an embedding
\cite{KP:crelle}.

\begin{proposition}\label{p:nucdimO2}
Let $A$ be a separable, exact $\Cstar$-algebra, and let $\phi \colon A \to \mathcal O_2$ be an
injective $^*$-homomorphism. Then
\begin{equation}
\dimnuc\phi = \left\{ \begin{array}{ll}
0 & \textrm{ if $A$ is quasidiagonal,} \\
1 & \textrm{ otherwise.}
\end{array} \right.
\end{equation}
\end{proposition}
\begin{proof}
As $\mathcal O_2$ is nuclear and $\mathcal O_\infty$-stable, so is $\phi$ by
\cite[Proposition~3.18]{Gabe17}. Hence $\dimnuc \phi \leq 1$ by Theorem~\ref{ThmC}.

First suppose that $\dimnuc\phi = 0$. Since the definitions of nuclear dimension zero and
decomposition rank zero coincide, $\dr\,\phi =0$ and thus $A$ is quasidiagonal by
Corollary~\ref{DetectQDCor}.

Now suppose that $A$ is quasidiagonal. Let $\omega$ be a free ultrafilter on $\mathbb N$, and let
$(\mathcal O_2)_\omega$ denote the norm ultraproduct of $\mathcal O_2$. Let $\iota \colon \mathcal O_2
\hookrightarrow (\mathcal O_2)_\omega$ be the canonical (unital) embedding. Then $\iota \circ \phi
\colon  A \to (\mathcal O_2)_\omega$ is nuclear and $\mathcal O_2$-stable.

Fix an isomorphism $\kappa \colon \mathcal O_2 \otimes \mathcal O_2 \rightarrow \mathcal O_2$ (the
existence of which is due to Elliott, but was first recorded in \cite{R:CRMPASC}; see
\cite[Theorem~5.2.1]{R-book}). By quasidiagonality of $A$ there exist integers
$(k_n)_{n=1}^\infty$ and cpc maps $\psi_n \colon A \to M_{k_n}$ such that
\begin{equation}
\psi \coloneqq (\psi_n)_{n=1}^\infty \colon A \to \prod_{n=1}^\infty M_{k_n} / \bigoplus_{n=1}^\infty M_{k_n}
\end{equation}
is an injective $^*$-homomorphism. Let $\rho_n \colon M_{k_n} \to \mathcal O_2$ be an embedding
for each $n$, and let
\begin{equation}
\eta_n \coloneqq  \kappa \circ (\rho_n \otimes 1_{\mathcal O_2}) \colon M_{k_n} \to \mathcal O_2.
\end{equation}
Then $(\eta_n \circ \psi_n)_\omega \colon A \to (\mathcal O_2)_\omega$ is an injective
$^*$-homomorphism with a cpc lift $(\eta_n \circ \psi_n)_{n=1}^\infty \colon A \to \prod_{n =1}^\infty \mathcal
O_2$. As $A$ is exact and $\mathcal O_2$ is nuclear, this cpc lift is nuclear by
\cite[Proposition~3.3]{D:JFA}, and thus $(\eta_n \circ \psi_n)_\omega$ is nuclear. To see that
$(\eta_n \circ \psi_n)_\omega$ is $\mathcal O_2$-stable, observe that the unital
$\Cstar$-subalgebra
\begin{equation}
\kappa(1_{\mathcal O_2} \otimes \mathcal O_2) \subseteq (\mathcal O_2)_\omega
\end{equation}
is isomorphic to $\mathcal O_2$, and commutes with the image of $(\eta_n \circ \psi_n)_\omega$.
Hence both $\iota \circ \phi \colon A \to (\mathcal O_2)_\omega$ and $(\eta_n \circ \psi_n)_\omega$ are
nuclear, $\mathcal O_2$-stable homomorphisms of $A$ into $(\mathcal O_2)_\omega$. Since $(\mathcal
O_2)_\omega$ is simple by \cite[Proposition~6.2.6]{R-book}, and since $\iota \circ \phi$ and $(\eta_n \circ
\psi_n)_\omega$ are injective, we trivially have $\mathcal I(\iota \circ \phi) = \mathcal I\big((\eta_n \circ
\psi_n)_\omega\big)$. Therefore Theorem~\ref{Thm:O2MvN} implies that $\iota \circ \phi$ and $(\eta_n \circ
\psi_n)_\omega$ are approximately Murray--von Neumann equivalent.

By \cite[Proposition~3.10]{Gabe17}, the maps
\begin{equation}
\iota \circ \phi \oplus 0, \, (\eta_n \circ \psi_n)_\omega \oplus 0 \colon A \to M_2(\mathcal O_2)_\omega
\end{equation}
are approximately unitary equivalent. Hence for $\mathcal F \subseteq A$ finite and $\epsilon >0$,
there exist a unitary $u\in M_2(\mathcal O_2)$ and an $n\in \mathbb N$ such that
\begin{equation}
\| \phi(a) \oplus 0 - \Ad u (\eta_n(\psi_n(a)) \oplus 0) \|  < \epsilon , \qquad a\in \mathcal F.
\end{equation}
Let $\eta \coloneqq  \Ad u \circ (\eta_n \oplus 0) \colon M_{k_n} \to M_2(\mathcal O_2)$ for each
$n$. Then $\eta \circ \psi_n$ approximates $\phi \oplus 0$ up to $\epsilon$ on $\mathcal F$. As
$\eta$ is a $^*$-homomorphism and $\psi_n$ is a cpc map, it follows that $\phi \oplus 0$ has
nuclear dimension zero. By Proposition~\ref{Prop:cpcHer}, so too does $\phi$.
\end{proof}

The following is essentially contained in \cite{PR:crelle}, and uses tricks of Blackadar and Cuntz
from \cite{BC:AJM}. Recall the definition of properly infinite, positive elements given above
Lemma~\ref{l:imageOinftystable}.

\begin{lemma}\label{l:fullpropinf}
Let $b\in B_+$ and $\epsilon >0$. Suppose that $b$ and $(b-\epsilon)_+$ are properly infinite and
generate the same ideal $I$ in $B$. Then $I$ contains a properly infinite, full projection.
\end{lemma}
\begin{proof}
In the following, we use the Cuntz relation on positive elements, and write $a\sim b$ if and only
if $a \precsim b$ and $b\precsim a$.  Arguing as in the last part of the proof of
\cite[Proposition~2.7 \mbox{(i)\;$\Rightarrow$\;(ii)}]{PR:crelle}, there is a positive element
$c_2 \in I$ for which $b \sim c_2$, and a projection $p\in I$ satisfying $pc_2 =
c_2$.\footnote{The following argument is essentially that of \cite[Proposition 2.7 \mbox{(i)\;
$\Rightarrow$\;(ii)}]{PR:crelle}: Fix $0 < \epsilon_1 < \epsilon$. As $b$ and $(b-\epsilon)_+$ are
properly infinite and generate the same ideal, it follows that $b \sim (b-\epsilon)_+ \sim (b-
\epsilon_1)_+$ by \cite[Proposition~3.5(ii)]{KR-AJM}. By \cite[Lemma~2.5(i) and
Proposition~3.3]{KR-AJM} we can find mutually orthogonal, positive elements $c_1$ and $c_2$ in
$\overline{(b-\epsilon_1)_+ I (b-\epsilon_1)_+}$ such that $(b-\epsilon)_+ \precsim c_1$ and
$(b-\epsilon)_+ \precsim c_2$. In particular, $b \precsim c_1$. By \cite[Remark~2.5]{PR:crelle}
there is a contraction $x\in I$ such that $x^\ast x (b - \epsilon_1)_+ = (b-\epsilon_1)_+$ and
$xx^\ast \in \overline{c_1 I c_1} \subseteq \overline{(b-\epsilon_1)_+ I (b-\epsilon_1)_+}$. Hence
$x$ is a scaling element in the sense of \cite{BC:AJM}, i.e.~$x$ is a contraction for which
$x^\ast x x x^\ast = x x^\ast$, and this $x$ satisfies $x^\ast x c_2 = c_2$ and $xx^\ast c_2 = 0$.
By \cite[Remark~2.4]{PR:crelle} there is a projection $p\in I$ such that $p c_2 = c_2$.} In
particular, $b \sim c_2 \precsim p$. As $b$ is properly infinite and $p\in I = \overline{BbB}$, it
follows that $p \precsim b$ by \cite[Proposition~3.5(ii)]{KR-AJM}. Hence $p$ and $b$ are Cuntz
equivalent, so $p$ is properly infinite and full in $I$.
\end{proof}

We can now prove Theorem~\ref{ThmD}: if $\theta$ is a full, $\mathcal O_2$-stable
$^*$-homomorphism out of a separable and exact $\Cstar$-algebra, then the nuclear dimension of
$\theta$ is~0 if $\theta$ is nuclear and $A$ is quasidiagonal, 1 if $\theta$ is nuclear and $A$ is
not quasidiagonal, and $\infty$ if $\theta$ is not nuclear.

\begin{proof}[Proof of Theorem~\ref{ThmD}]
If $\theta$ is not nuclear, then certainly $\dimnuc \theta=\infty$. So suppose $\theta$ is
nuclear. Then by Theorem~\ref{ThmC} we have $\dimnuc \theta \leq 1$. So it suffices to show that
$\dimnuc \theta = 0$ if and only if $A$ is quasidiagonal.

First suppose that $\dimnuc\theta = 0$. Then $\dr\,\theta =0$ and thus $A$ is quasidiagonal by
Corollary~\ref{DetectQDCor}.

Now suppose that $A$ is quasidiagonal. Let $a\in A$ be positive and non-zero, and $0 < \epsilon <
\| a\|$. As $\theta$ is $\mathcal O_2$-stable it is $\mathcal O_\infty$-stable. So as $\theta$ is also full, $\theta(a)$ and $(\theta(a)-\epsilon)_+$
are both properly infinite by Lemma~\ref{l:imageOinftystable}, and they both generate $B$ as an
ideal. By Lemma~\ref{l:fullpropinf}, $B$ contains a full, properly infinite projection. Hence
there exists an embedding $\iota \colon \mathcal O_2 \to B$ such that $\iota(1_{\mathcal O_2})$ is
a full projection in $B$. Let $\phi \colon A \to \mathcal O_2$ be an injective $^*$-homomorphism
as given by Kirchberg's embedding theorem. Then both $\theta$ and $\iota \circ \phi$ are nuclear,
full, $\mathcal O_2$-stable $^*$-homomorphisms, so they are approximately Murray--von Neumann
equivalent by Theorem~\ref{Thm:O2MvN}. By \cite[Proposition~3.10]{Gabe17}, $\theta \oplus 0,
(\iota \circ \phi)\oplus 0 \colon A \to M_2(B)$ are approximately unitary equivalent, and in
particular have the same nuclear dimension. By Proposition~\ref{p:nucdimO2}, $\phi$ has nuclear
dimension zero, and thus so does $\theta \oplus 0$. Hence $\theta$ has nuclear dimension zero by
Proposition~\ref{Prop:cpcHer}.
\end{proof}

\begin{problem}
Characterise when $^*$-homomorphisms have nuclear dimension zero.
\end{problem}

\end{document}